\documentclass[11pt]{amsart}
\usepackage{amsmath, amsfonts, amssymb}
\usepackage{color}
\usepackage{graphicx}

\newtheorem{thm}{Theorem}
\theoremstyle{definition}

\newtheorem{cor}[thm]{Corollary}

\newtheorem{ex}[thm]{Example}
\newtheorem{lem}[thm]{Lemma}
\newtheorem{prop}[thm]{Proposition}
\newtheorem{rem}[thm]{Remark}

\newcommand{\R}{\mathbb{R}}
\newcommand{\N}{\mathbb{N}}

\begin{document}

\title{A Study of Projections of 2-Bouquet Graphs}
\author{Elaina Aceves}
\address{Department of Mathematics, California State University, Fresno, CA 93740}
\email{ekaceves516@gmail.com}

\date{}


\begin{abstract}
We extend the concepts of trivializing and knotting numbers for knots to spatial graphs and 2-bouquet graphs, in particular. Furthermore, we calculate the trivializing and knotting numbers for projections and pseudodiagrams of 2-bouquet spatial graphs based on the number of precrossings and the placement of the precrossings in the pseudodiagram of the spatial graph.
\end{abstract}

\maketitle


\section{Introduction}

\subsection{Mathematical Knots}

A \textbf{knot} $K$ is an embedding of a circle into $\mathbb{R}^3$ and a \textbf{diagram} of a knot $K$ is a projection of $K$ onto a plane with transverse double points together with over/under crossing information.
Two knots are \textbf{ambient isotopic} if and only if their diagrams are related by a finite sequence of the Reidemeister moves given in Figure~\ref{fig:ReidMoves}.

\begin{figure}[ht]
\begin{center}
\includegraphics[height=4.5cm]{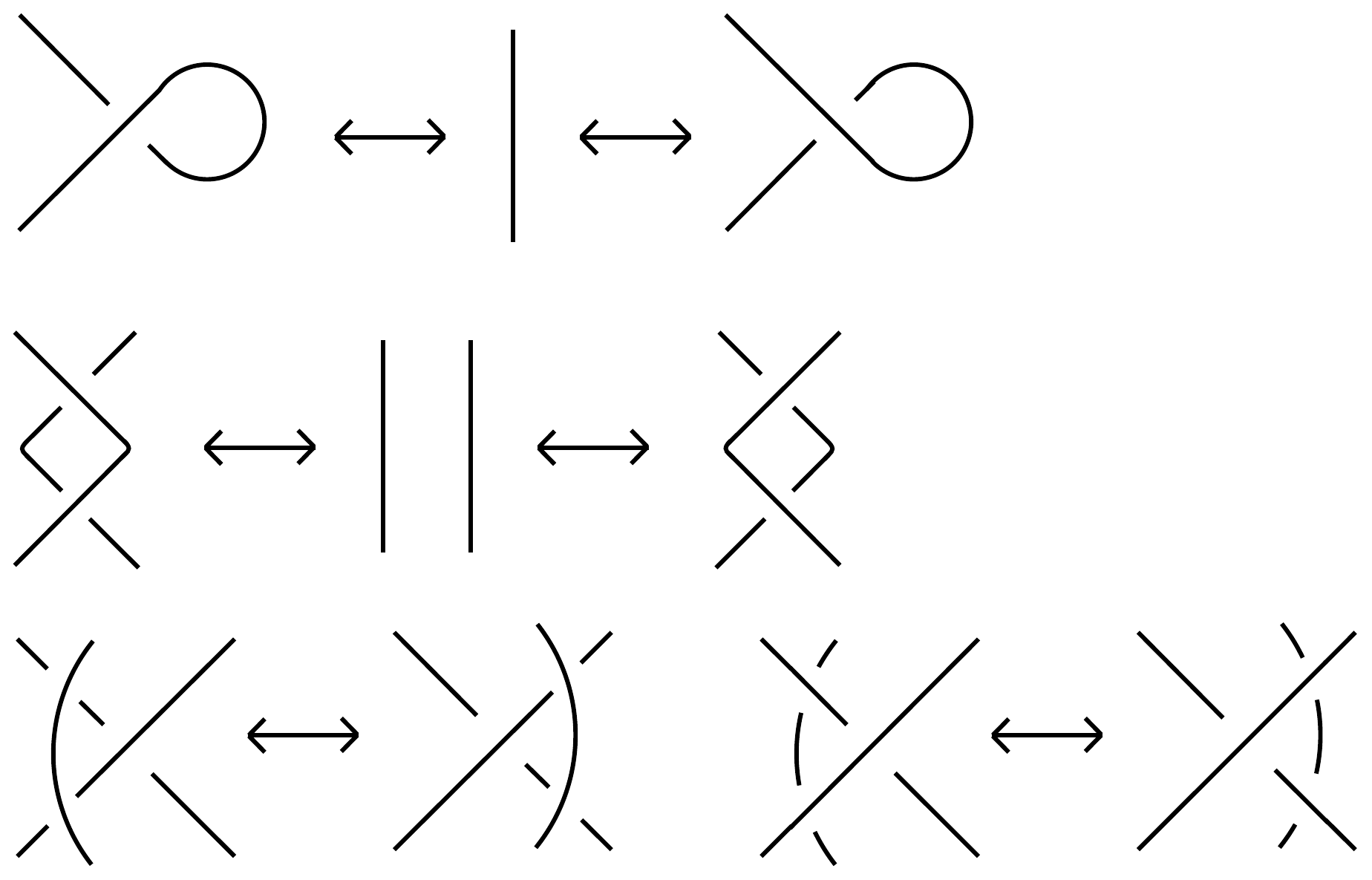}
\put(-230, 105){\fontsize{15}{11}RI:}
\put(-230, 60){\fontsize{15}{11}RII:}
\put(-230, 15){\fontsize{15}{11}RIII:}
\caption{The Reidemeister Moves}
\label{fig:ReidMoves}
\end{center}
\end{figure}

A knot is called \textbf{trivial}, or \textbf{unknotted}, if any diagram of the knot can be transformed via a finite sequence of the Reidemeister moves to a simple closed loop in a plane. Otherwise, a knot is called \textbf{nontrivial}, or \textbf{knotted}. We note that there is only one trivial knot up to ambient isotopy, which we call the \textbf{unknot}.
A more detailed introduction to knots can be found, for example, in \cite{A} and  \cite{K3}.

\subsection{Pseudodiagrams of Knots}

For application purposes, one may want to allow the possibility that no information is known about which strand lies over the other at a double point in a diagram of a knot. We refer to this type of double point as a \textbf{precrossing}.
A \textbf{projection} $P$ is a knot diagram without over/under information at every double point in the diagram, so all of the double points in a projection are precrossings. A \textbf{pseudodiagram} $Q$ of a knot is a projection $P$ in which over/under information may be known at some of the precrossings of $P$. 
To be precise, a pseudodiagram can contain both crossings and precrossings.
With these definitions in place, all projections are pseudodiagrams but not all pseudodiagrams are projections. 
In Figure~\ref{fig:Pseudo}, diagrams (a) and (b) are both pseudodiagrams but only diagram (a) is a projection.

\begin{figure}[ht]
\begin{center}
\includegraphics[height=3.0cm]{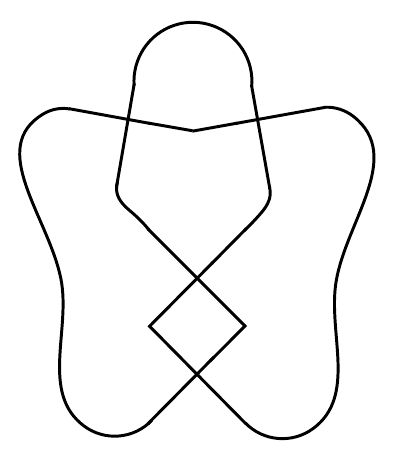} \quad 
\includegraphics[height=3.0cm]{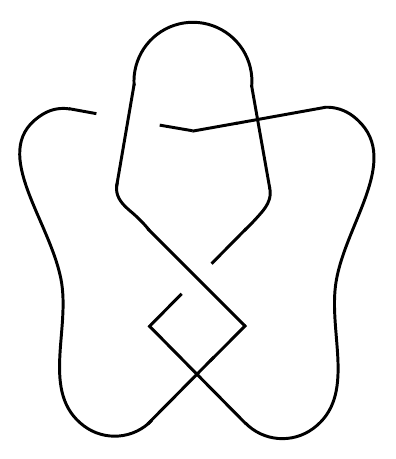}
\put(-133, -10){\fontsize{15}{11}(a)}
\put(-44, -10){\fontsize{15}{11}(b)}
\caption{Pseudodiagrams}
\label{fig:Pseudo}
\end{center}
\end{figure}

\textbf{Resolving} a precrossing in a projection or a pseudodiagram is the action of replacing the precrossing with a crossing of either type. 
By resolving a precrossing of a pseudodiagram $Q$, we obtain two new pseudodiagrams $Q_1$ and $Q_2$, one for each way that we can obtain a crossing from a precrossing.
A \textbf{resolution} of $Q$ is a knot diagram that is obtained by resolving all of the precrossings of $Q$ into crossings.
The diagram shown in Figure~\ref{fig:Pseudo2} is a resolution of the pseudodiagram (b) from Figure~\ref{fig:Pseudo}.

\begin{figure}[ht]
\begin{center}
\includegraphics[height=3.0cm]{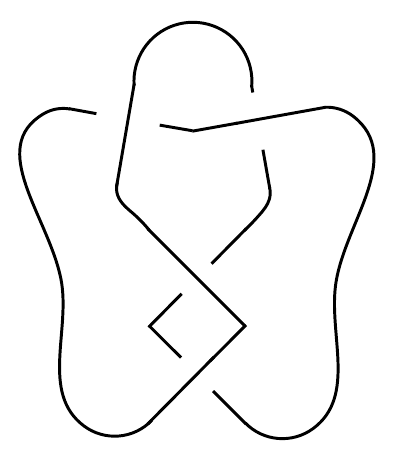}
\caption{Resolution of a Pseudodiagram}
\label{fig:Pseudo2}
\end{center}
\end{figure}

A pseudodiagram $Q$ is \textbf{trivial} if \textit{any} resolution of $Q$ represents the trivial knot.
Conversely, a pseudodiagram $Q$ is \textbf{knotted} if \textit{any} resolution of $Q$ represents a nontrivial knot.
We remark that a pseudodiagram can be trivial, knotted, or neither. 
In Figure~\ref{fig:TKN},  pseudodiagram (a) is trivial, (b) is knotted, and (c) is neither.
Note that the above definitions apply to projections as well.

\begin{figure}[ht]
\begin{center}
\includegraphics[height=3.0cm]{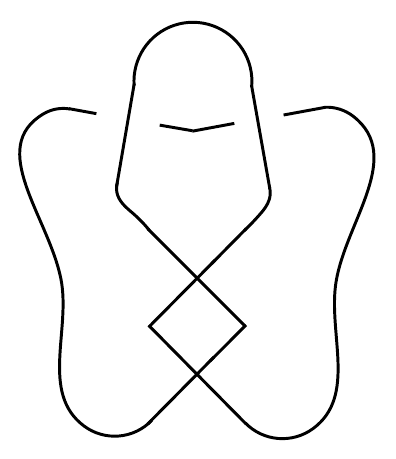} \quad 
\includegraphics[height=3.0cm]{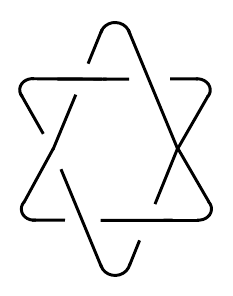} \quad
\includegraphics[height=3.0cm]{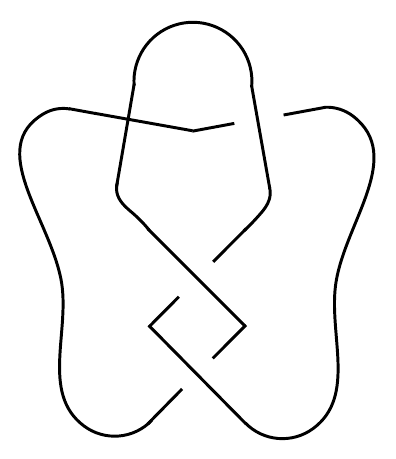}
\put(-129, -10){\fontsize{15}{11}(b)}
\put(-43, -10){\fontsize{15}{11}(c)}
\put(-214, -10){\fontsize{15}{11}(a)}
\caption{Pseudodiagrams that are Trivial, Knotted, or Neither }
\label{fig:TKN}
\end{center}
\end{figure}

The \textbf{trivializing number} of a knot projection $P$, denoted $tr(P)$, (or \textbf{knotting number}, denoted $kn(P)$, respectively) is the minimum number of precrossings that need to be replaced by a crossing in order to obtain a trivial pseudodiagram (or knotted pseudodiagram, respectively).  
If the trivializing or knotting number does not exist, we simply denote it as $\infty$.

The majority of this work is based on ideas from Hanaki's paper~\cite{H}, which focuses on calculating the trivializing and knotting numbers for projections of knots.
In this paper, we extend Hanaki's approach from projections of knots to projections of spatial 2-bouquet graphs.

\section{Rigid Vertex Embeddings of Spatial 2-Bouquet Graphs}

A \textbf{spatial graph} is an embedding of a graph in $\R^3$ while a \textbf{diagram} of a spatial graph $G$ is a projection of $G$ onto a plane.
One of the simplest graphs to investigate are the bouquet graphs, which are graphs with one vertex and only loops as edges. The \textbf{2-bouquet graph} is the graph with one vertex and two loops, as shown in Figure~\ref{fig: K and L}.
Throughout this paper, we will refer to the loops of a 2-bouquet graph as petals. 

We consider only rigid-vertex embeddings of 4-valent graphs. Specifically, we regard a spatial graph as an embedding in $\R^3$ of a 4-valent graph whose vertices have been replaced by rigid disks. Each disk has four strands attached to it, and there is a cyclic order of these strands which is determined by the rigidity of the disk. 

Two 4-valent spatial graphs, $G_1$ and $G_2$, with rigid vertices are called \textbf{ambient isotopic} if there exists an orientation-preserving homeomorphism of $\mathbb{R}^3$ onto itself that maps $G_1$ to $G_2$. It is well-known that $G_1$ and $G_2$ are ambient isotopic if and only if there is a finite sequence of extended Reidemeister moves transforming a diagram of $G_1$ into a diagram of $G_2$. 
The \textbf{extended Reidemeister moves} are depicted in Figure~\ref{fig: Ext}; the solid dot in moves RIV and RV represents a 4-valent vertex of the graph. 
These moves introduce an equivalence relation on the diagrams, and as a consequence, we can view a spatial graph as the equivalence class of a spatial graph diagram.
We refer the reader to Kauffman's work~\cite{K1,K3} for more details on rigid-vertex embeddings of graphs.

\begin{figure}[ht] \begin{center}
\includegraphics[height=4.5cm]{Reid.pdf}
\put(-240, 105){\fontsize{15}{11}RI:}
\put(-240, 60){\fontsize{15}{11}RII:}
\put(-240, 15){\fontsize{15}{11}RIII:}
\end{center}
\begin{center}
\includegraphics[height=1.5cm]{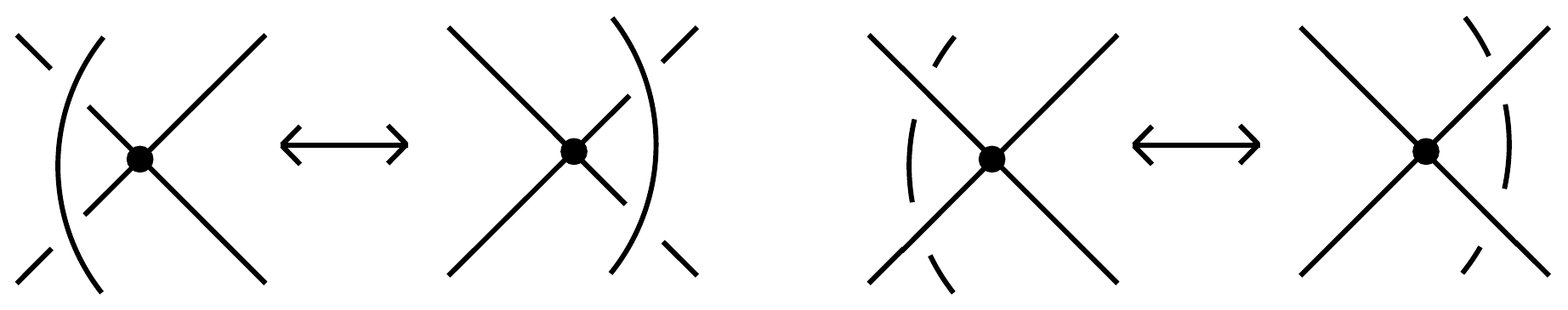} 
\put(-250, 15){\fontsize{15}{11}RIV:}
\end{center}
\begin{center} 
\includegraphics[height=3.0cm]{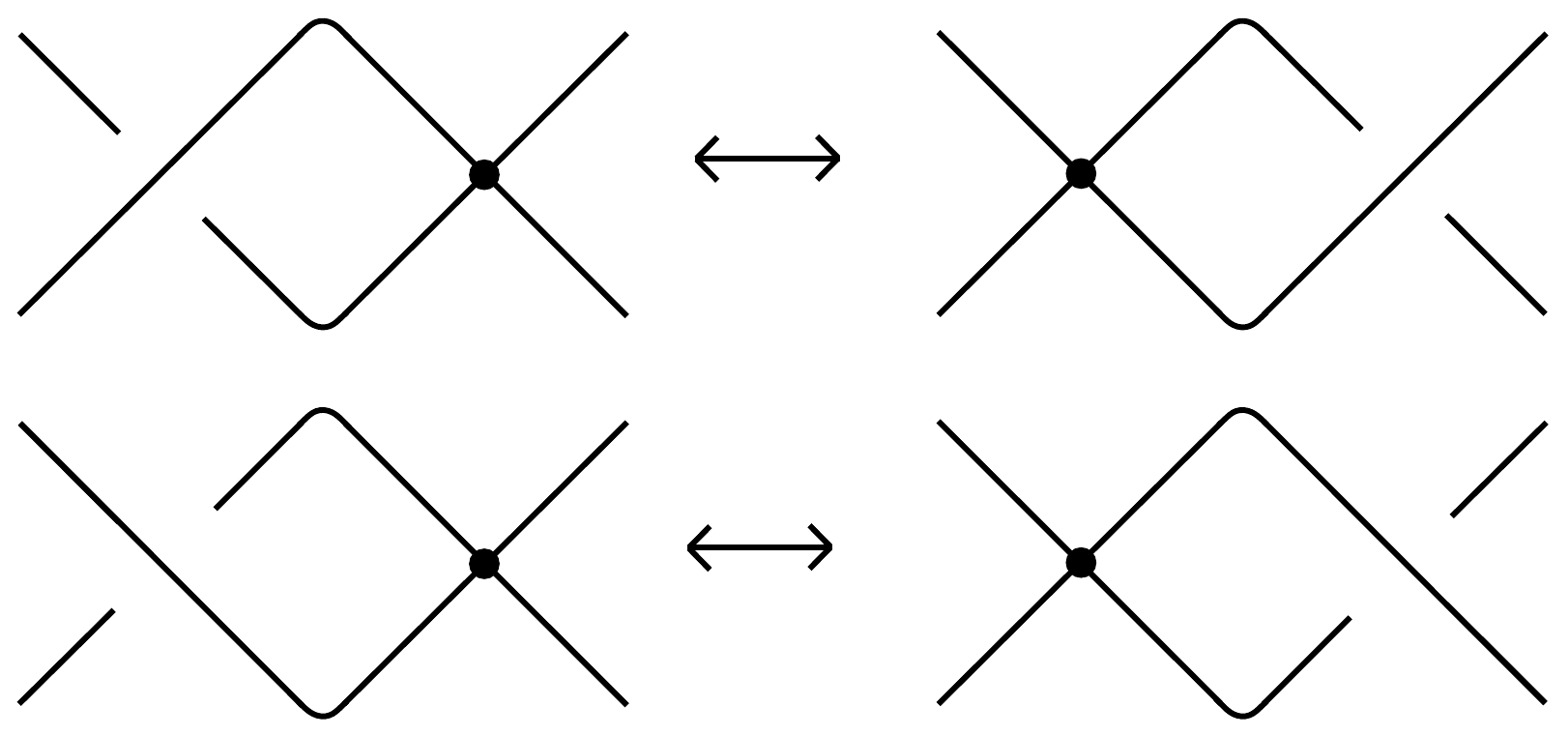} 
\put(-233, 60){\fontsize{15}{11}RV:}
\end{center}
\caption{The Extended Reidemeister Moves}
\label{fig: Ext}
\end{figure}

We call two spatial graph diagrams \textbf{equivalent} if one can be transformed into the other via a finite sequence of the extended Reidemeister moves. Therefore, equivalent diagrams belong to the same equivalence class.

In this paper, we restrict our attention to rigid-vertex embeddings of the 2-bouquet graph, and we  call such an object a \textbf{2-bouquet} for short. Moreover, we consider both pseudodiagrams and projections of 2-bouquets.

The \textbf{mirror image} of a 2-bouquet pseudodiagram $D$ is the pseudodiagram $D^*$ obtained from $D$ by changing the overcrossings into undercrossings and vice-versa, for all crossings of $D$. 
The precrossings of $D$ remain unchanged.

There are two trivial rigid-vertex embeddings of the 2-bouquet in $\mathbb{R}^3$, up to cyclic order of the edges meeting at the vertex, as shown in Figure~\ref{fig: K and L}. We will refer to these as the \textbf{unknotted 2-bouquet of type $K$} and the \textbf{unknotted 2-bouquet of type $L$}, respectively (we have borrowed some terminology from Oyamaguchi's work ~\cite{O}). For simplicity, we also refer to these as the \textbf{trivial 2-bouquet} of type $K$ or of type $L$.

\begin{figure}[ht] \label{K and L}
\begin{center}
\includegraphics[height=5.5cm]{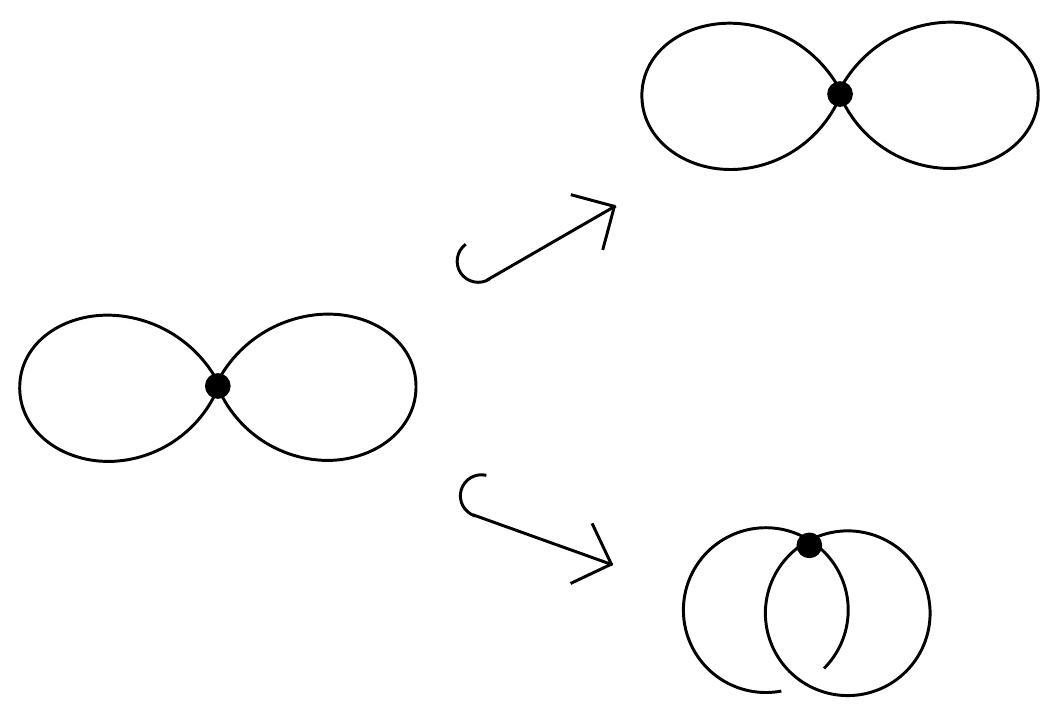} 
\put(-65, 95) {type $K$}
\put(-65, -10) {type $L$}
\put(-62, 154){\fontsize{8}{11}1}
\put(-62, 109){\fontsize{8}{11}2}
\put(-32, 109){\fontsize{8}{11}3}
\put(-32, 154){\fontsize{8}{11}4}
\put(-75, 45){\fontsize{8}{11}1}
\put(-74, 20){\fontsize{8}{11}2}
\put(-44, 20){\fontsize{8}{11}3}
\put(-45, 45){\fontsize{8}{11}4}
\end{center}
\caption{Trivial 2-Bouquets of Type $K$ and Type $L$}
\label{fig: K and L}
\end{figure}
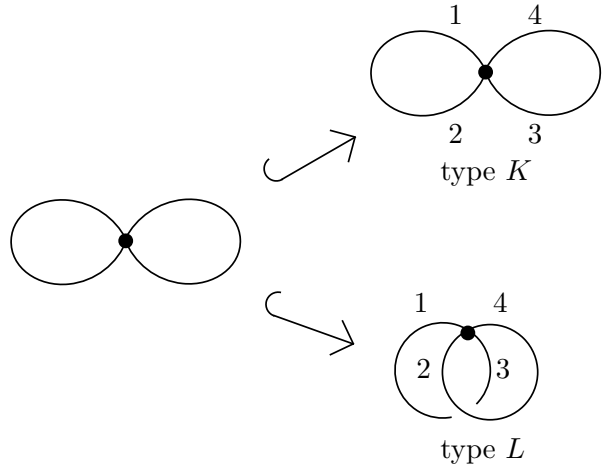

To identify a 2-bouquet diagram as type $K$ or type $L$, we exit the vertex along a given edge and by traveling along the petal, return to the vertex at a different edge.
If the two edges are adjacent in the diagram, we have a 2-bouquet diagram of type $K$.
Otherwise, we have a 2-bouquet diagram of type $L$.

We call a pseudodiagram of a 2-bouquet \textbf{$K$-trivial} (or \textbf{$L$-trivial}, respectively) if any diagram obtained by resolving all of its precrossings is equivalent to the standard diagram of the unknotted 2-bouquet of type $K$ (or type $L$, respectively). Otherwise, a pseudodiagram of a 2-bouquet  is called \textbf{knotted}.

Let $Q$ be a pseudodiagram of a 2-bouquet. 
The \textbf{trivializing number} of $Q$, denoted $tr(Q)$, (or \textbf{knotting number} of $Q$, denoted $kn(Q)$, respectively), is the minimum number of precrossings that needs to be replaced by a crossing to obtain a $K$-trivial or $L$-trivial pseudodiagram (or knotted pseudodiagram, respectively).  
If the trivializing or knotting number does not exist, we denote it as $\infty$.
It is clear from the previous definition that the trivializing and knotting numbers are non-negative integers or infinity.

We will frequently consider the diagram shown in Figure~\ref{fig: kcirc}, which we refer to as a \textbf{pretzel projection}, where $x_i \in \N$ for all $1 \leq i \leq k$ and where $k$ is finite positive integer. The leftmost dashed circle contains only a 4-valent vertex, and the remaining $k$ dashed circles from left to right contain $x_1$, $x_2, \, \dots\, , x_k$ precrossings stacked vertically.
We denote such a diagram as $(x_1, x_2,\, \dots\, , x_k)$. 

\begin{figure}[ht]
\begin{center}
\includegraphics[height=2.5cm]{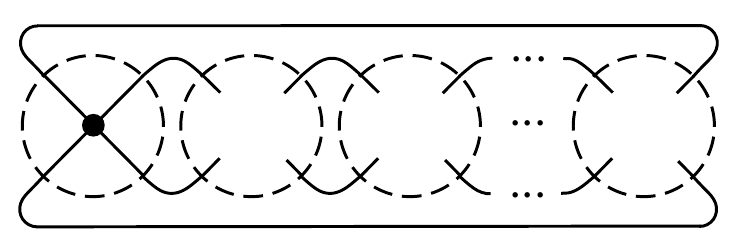}
\put(-145, 31){\fontsize{10}{11}$x_1$}
\put(-100, 31){\fontsize{10}{11}$x_2$}
\put(-33, 31){\fontsize{10}{11}$x_k$}
\quad 
\raisebox{5pt}{\includegraphics[height=2.0cm]{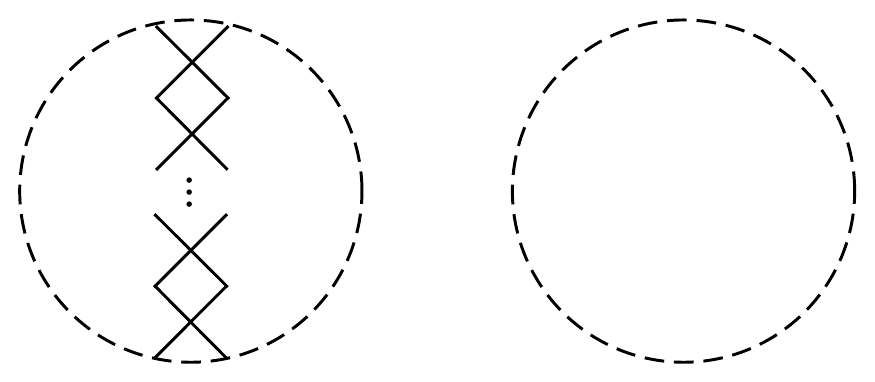}}
\put(-33, 30){\fontsize{10}{11}$x_i$}
\put(-70, 30){\fontsize{10}{11}$=$}
\end{center}
\caption{Pretzel Projection $(x_1, x_2, \, \dots \, , x_k)$}
\label{fig: kcirc}
\end{figure}
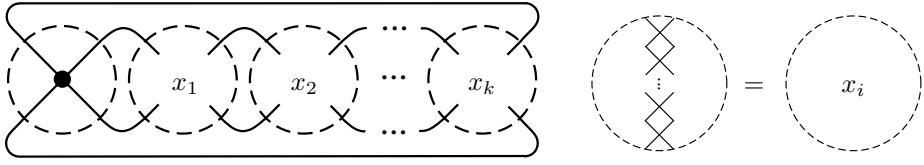

\begin{rem}
We want to emphasize that only one of the dashed circles contains the 4-valent vertex. 
Also, if we begin with a diagram where the 4-valent vertex is not in the leftmost circle, we can use planar isotopy to transform a diagram into another diagram where the 4-valent vertex is in the leftmost circle.
\end{rem}

\begin{rem}
We will consider the projection depicted in Figure~\ref{fig: kcirc} as equivalent to the diagram in Figure~\ref{fig: kcirc2}.

\begin{figure}[ht]
\begin{center}
\includegraphics[height=2.5cm]{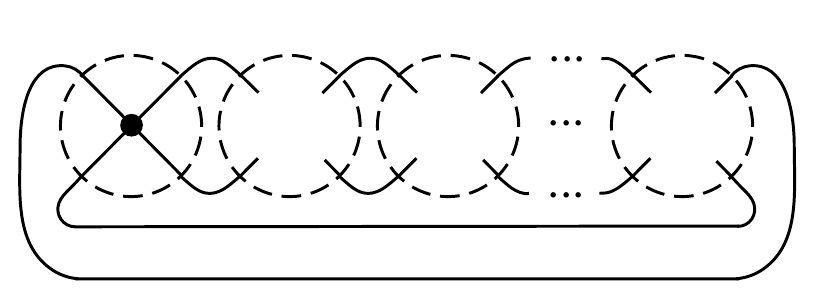}
\put(-130, 38){\fontsize{10}{11}$x_1$}
\put(-92, 38){\fontsize{10}{11}$x_2$}
\put(-36, 38){\fontsize{10}{11}$x_k$}
\end{center}
\caption{Equivalent Pretzel Projection $(x_1, x_2, \, \dots \, , x_k)$}
\label{fig: kcirc2}
\end{figure}
\end{rem}

Throughout the remainder of the paper, we will refer to the $x_i$ precrossings that are stacked vertically in each dashed circle as a \textbf{stack}. 
We label the first (top) precrossing in the stack of $x_i$ precrossings as $p_{i,1}$ and the last (bottom) precrossing in the stack as $p_{i, x_i}$.

We proceed to our first result.

\begin{lem} \label{Even}
Given a pretzel projection $(x_1, x_2,\, \dots\, , x_k)$ of a 2-bouquet, at most one $x_i$ is even.
\end{lem}

\begin{proof}
We prove this result by contradiction.
Suppose that $x_i$ and $x_j$ are both even for some $i \neq j$.
Without loss of generality, let $i <j$. 
Notice that we have another component in our diagram that joins stacks $x_i$ and $x_j$ by traveling through the top right hand corner of $p_{i,1}$, the bottom right hand corner of $p_{i, x_i}$, through the stacks $x_{i+1}, \cdots, x_{j-1}$, through the bottom left hand side of $p_{j, x_j}$, and the top left hand side of $p_{j, 1}$. 
Since we have an additional component in union with our 2-bouquet, we do not have a projection of a 2-bouquet.  
Thus, we can only allow one of our stacks of precrossings to contain an even number of precrossings.  
\end{proof}

\begin{lem}\label{Singletypes}
Given a pretzel projection $P=(x_1, x_2, \cdots, x_k)$, the diagram is of type $K$ if and only if one $x_j$ is even or all $x_i$ are odd and $k$ is even.
The diagram is of type $L$ if and only if all $x_i$ are odd and $k$ is odd.
\end{lem}

\begin{proof}
We prove only one implication of each of the statements to avoid repetition, as the other implication follows similarly. 
We prove the lemma by cases and label the edges around the vertex as shown in Figure~\ref{fig:labeled_vertex}.
Without loss of generality, let the first petal begin by exiting the vertex along edge 2.

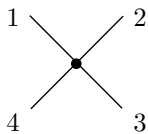
\begin{figure}[ht]
\[
\includegraphics[height=1.5cm]{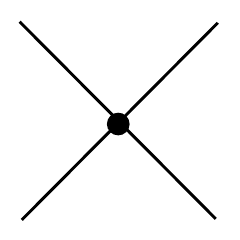}
\put(-48, 35){\fontsize{10}{11}$1$}
\put(-48, -5){\fontsize{10}{11}$4$}
\put(0, 35){\fontsize{10}{11}$2$}
\put(0, -5){\fontsize{10}{11}$3$}
\]
\caption{Labeled Edges on a Vertex}\label{fig:labeled_vertex}
\end{figure}

Case 1: The stack $x_j$ is even for some $1 \leq j \leq k$.

We have that the edge labeled 2 enters the $j$th stack either via the top left hand strand of precrossing $p_{j,1}$ or in the bottom left hand strand of precrossing $p_{j,x_j}$. 
Since $x_j$ is even, the strand through the top left hand corner of precrossing $p_{j,1}$ is connected to the strand in the bottom left hand corner of precrossing $p_{j,x_j}$. 
When strand 2 exits the $j$th stack, it connects with edge 3 upon entering the vertex.
Since edges 2 and 3 are adjacent to each other, we have that $P$ is of type $K$.

Case 2: Suppose that all $x_i$ are odd and $k$ is even. 

Since $x_1$ is odd, strand 2 travels through the first stack of precrossings and exits the stack through the bottom right hand strand of precrossing $p_{1,x_1}$ and enters the second stack through the bottom left hand strand of precrossing $p_{2,x_2}$.
Since $x_2$ is odd, edge 2 travels through the second stack of precrossings and exits the stack through the top right hand strand of precrossing $p_{2,1}$.
This process continues through the $k$ stacks of our diagram. 
Since $k$ is even, strand 2 exits the final stack through the top right hand strand of precrossing $p_{k,1}$.
Hence, the edge labeled 2 enters the vertex through the edge labeled 1.
Since these edges are adjacent to each other, we have that $P$ is of type $K$.

Case 3: Suppose that all $x_i$ are odd and $k$ is odd. 

Using the same reasoning as in the previous case, we know that strand 2 exits the last stack of precrossings through the bottom right hand strand of precrossing $p_{k,x_k}$.
Thus, strand 2 enters the vertex through strand 4.
Since strands 2 and 4 are not adjacent to each other, we have that $P$ is of type $L$.
\end{proof}

\section{Trivializing and Knotting Numbers for Pretzel Projections}
\subsection{Preliminary Results}

In this section, we begin the exploration of calculating the trivializing and knotting numbers of our 2-bouquets.

\begin{lem}\label{kn1}
Let $P$ be a projection of a 2-bouquet. Then $kn(P) \geq 2$.
\end{lem}

\begin{proof} We prove the statement using contradiction.

Suppose first that $kn(P)=0$.
Notice that we can resolve the precrossings of $P$ (if any) to force one petal to lie above the other, forcing the diagram to be trivial. 
This is similar to the method of unknotting a knot via the crossings-change operation so as to obtain an ascending knot diagram. 
Therefore, at least one resolution of $P$ is always trivial, a contradiction to $P$ being knotted.
Therefore, $kn(P) \neq 0$.

Suppose that $kn(P)=1$. 
This implies that we only need to resolve one precrossing in $P$ with a certain type of crossing to obtain a knotted pseudodiagram. 
Let $Q^+$ be the knotted pseudodiagram obtained from $P$ by resolving the one precrossing with the appropriate crossing.
Furthermore, let $Q^-$ be the pseudodiagram obtained from $P$ by resolving the one precrossing with the other type of crossing. 
Let $D$ be any diagram obtained from $Q^-$ by resolving all of the precrossings in $Q^-$. 
Notice that the mirror image diagram $D^*$ can be obtained from $Q^+$, and thus $D^*$ is knotted. 
Since the mirror image of a knotted spatial graph is also knotted, we have that $D$ is a knotted 2-bouquet.
Because this is true for any diagram $D$, we have that $Q^-$ is knotted. 
Since both $Q^+$ and $Q^-$ are knotted, it is not required to resolve the single precrossing with a particular type. 
Therefore, $kn(P)=0$, which we already know is a contradiction. 

Hence, $kn(P) \geq 2$.
\end{proof}

A stack of at least two double points, where there are crossings or precrossings, is considered a \textbf{knotted stack} if, regardless of how the precrossings of the stack are resolved, the two strands involved in the stack are linked and cannot be separated with a series of Reidemeister II moves.

\begin{ex}
Of the two stacks in Figure~\ref{fig: 2stacks}, only the left stack is knotted because regardless of how the precrossing is resolved, the two strands are always linked together.
On the other hand, we can resolve the two precrossings in the right stack where the overstrand of the resulting crossings has a negative slope, so as to allow two consecutive Reidemeister II moves (starting in the middle of the diagram) to separate the two strands.
Thus, the right stack is not knotted, because there exists a resolution of the precrossings that results in the separation of the two strands.

\begin{figure}[ht]
\[
\includegraphics[height=3.5cm]{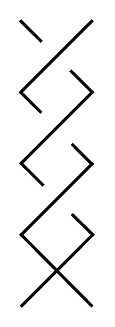} \quad \quad  \includegraphics[height=3.5cm]{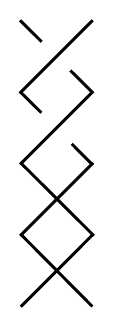} 
\]
\caption{Two Stacks}
\label{fig: 2stacks}
\end{figure}
\end{ex}

\begin{lem}\label{sk}
Given a pretzel pseudodiagram $Q=(x_1, x_2, \cdots, x_k)$ where we allow crossings as well as precrossings in $Q$, if any stack of $Q$ is knotted, then the pseudodiagram $Q$ is knotted. 
\end{lem}

\begin{proof}
Assume that the stack of $Q$ that contains $x_i$ double points (where $x_i \geq 2$) is knotted.

If $x_i$ is even, then the two petals only interact in that particular stack of $x_i$ double points.
Since the stack containing $x_i$ double points is knotted, the two petals are linked together regardless of how the remaining precrossings are resolved, forcing $Q$ to be knotted. 

Next, we consider the case when $x_i$ is odd.  
In an attempt to force our diagram to be trivial, we can resolve the remaining precrossings in the other stacks to create as many instances of the Reidemeister II move as we see possible.
However, with the rigidity of the 4-valent vertex and the way in which we construct our pretzel diagrams, a Reidemeister II move cannot be implemented by using one crossing in the stack of $x_i$ double points and another crossing in an adjacent stack, unless both stacks only contain a single crossing.
Since the stack of $x_i$ double points contains at least two double points,
we cannot implement a Reidemeister II move between the two adjacent stacks.
As a result, the stack of $x_i$ double points remains knotted after all of the remaining precrossings in the other stacks have been resolved. 
Consequently, the two petals are linked together within the stack of $x_i$ double points, and we have that $Q$ is knotted.
\end{proof}

We proceed by exploring the trivializing number.
It was shown in~\cite{H} that the trivializing number for projections of knots is even. 
This is the case for projections of 2-bouquets, which we will now show. 

\begin{lem} \label{tr1}
Let $P$ be a projection of a 2-bouquet. Then the trivializing number of $P$ is even.
\end{lem}

\begin{proof}
Recall that our goal with calculating the trivializing number is to find the minimum number of precrossings that need to be changed to crossings, in order to obtain a trivial pseudodiagram. 
Thus, we need to resolve the precrossings in such a way so that there exists a sequence of extended Reidemeister moves transforming our original diagram into a trivial pseudodiagram. 

First, we consider the Reidemeister I move. 
If we want to use the Reidemeister I move, we need to have a precrossing in $P$ that has two adjacent strands joined together. 
Regardless of how we resolve the corresponding precrossing, we can always use a Reidemeister I move to simplify the diagram.
Thus, we do not need to resolve the precrossing in a certain manner to make the diagram trivial.
As a result, utilizing a Reidemeister I move does not affect the trivializing number. 

Next we consider the Reidemeister II move. 
To implement this move, we need to resolve two precrossings in $P$ that are adjacent to each other (and involve two parallel strands of the projection) in such a way as to force one strand to be above the other. 
Therefore, the trivializing number increases by two every time we need to apply a Reidemeister II move.

Consider the Reidemeister III move. 
To accomplish a Reidemeister III move, we need to resolve two precrossings that correspond to the strand that we are attempting to slide over or under the remaining crossing. 
As a result, the trivializing number increases by two when we implement a Reidemeister III move.
Using similar reasoning, the trivializing number increases by two for an RIV move as well, to ensure that the strand slides over or under the 4-valent vertex involved in the move. 

Lastly, we consider the RV move (see Figure~\ref{fig: Ext}).
This move has a representation as shown in Figure~\ref{fig: R5alt}.

\begin{figure}[ht]
\[ 
\includegraphics[height=4.0cm]{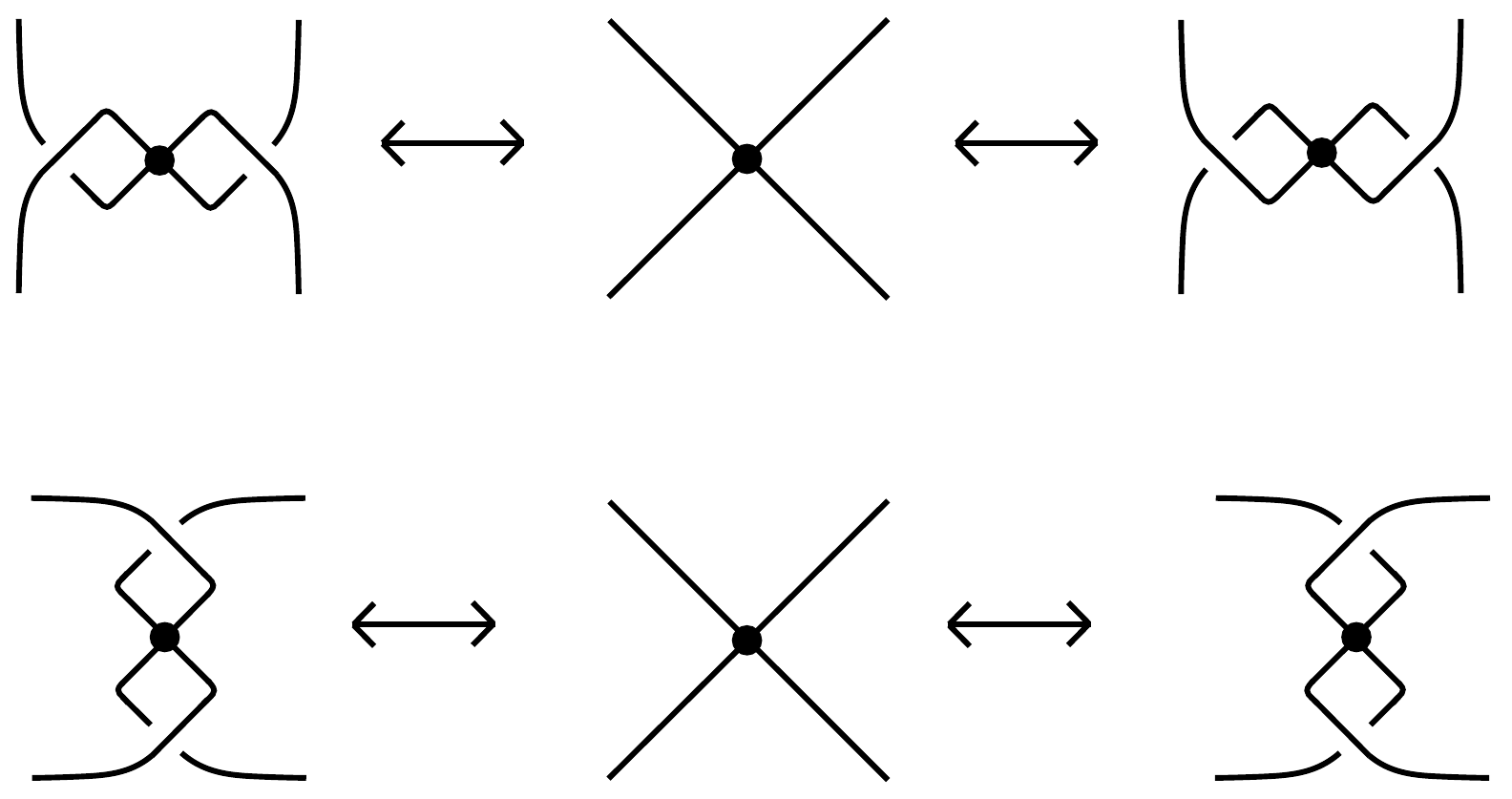} 
\put(-250, 50){\fontsize{16}{11}RV:}
\] 
\caption{Alternate Version of the  RV Move}
\label{fig: R5alt}
\end{figure}

Indeed, Figure~\ref{fig: R5alt2} shows that the first two diagrams depicted in Figure~\ref{fig: R5alt} are equivalent via an RIV move and a Reidemeister II move. The other cases of the alternate version of the RV move are verified similarly.

\begin{figure}[ht]
\[ 
\includegraphics[height=1.5cm]{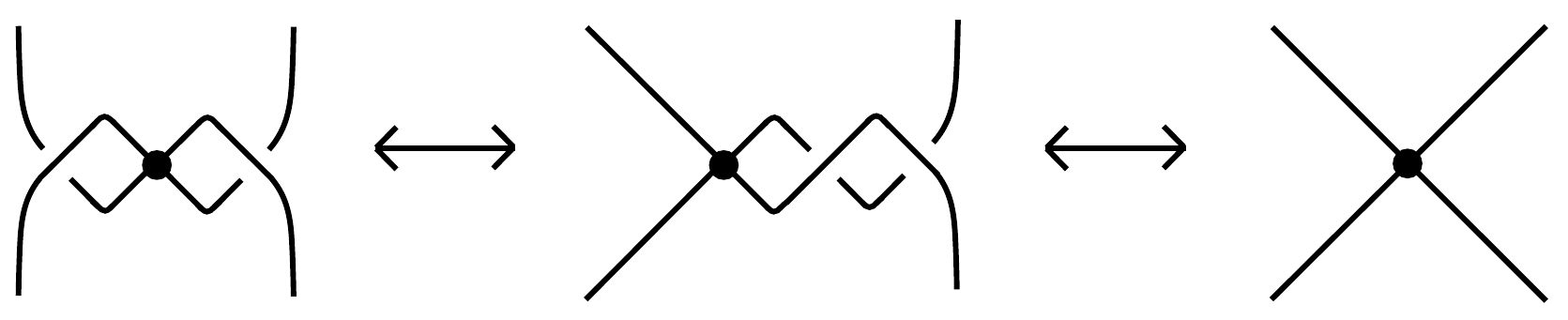} 
\put(-160, 40){\fontsize{15}{11}RIV}
\put(-70, 40){\fontsize{15}{11}RII}
\] 
\caption{Proof of a Case of the Alternate Version of the RV Move}
\label{fig: R5alt2}
\end{figure}

To employ the RV move, we need to resolve two precrossings in $P$ so that two of the adjacent strands exiting the vertex involved in the move become overstrands at the crossings while the remaining two adjacent strands exiting the vertex become understrands at the crossings.   
Thus, to apply the RV move, the trivializing number increases by two.  

Regardless of which one of the extended Reidemeister moves we employ, the trivializing number either remains the same or increases by two. Consequently, the trivializing number is even.
\end{proof}

Before we can progress to the next result, we will define the diagram  \textbf{\textit{T}(\textit{p,q})} as the diagram in Figure~\ref{fig: Tpq} which contains $p$ and $q$ precrossings where $p,q \geq 0$.

\begin{figure}[ht]
\[
\raisebox{45pt}{$T(p,q)$ = } \includegraphics[height=2.5cm]{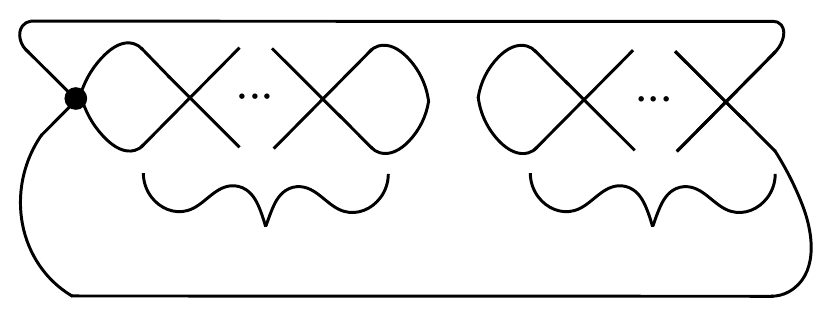}
\put(-131, 12){\fontsize{10}{11}$p$}
\put(-44, 12){\fontsize{10}{11}$q$}
\]
\caption{Diagram $T(p,q)$}
\label{fig: Tpq}
\end{figure}

\begin{lem} \label{tr2}
Let $P$ be a pretzel projection of a 2-bouquet and let $Q$ be a pseudodiagram obtained from $P$ by resolving some of the precrossings of $P$ into crossings followed by any possible applications of the Reidemeister moves.
Then $tr(Q)=0$ if and only if $Q=(1)$ or $Q=T(p,q)$ for some $p,q \geq 0$.
\end{lem}

\begin{proof}
$(\Leftarrow)$ It is clear that  $tr(Q)=0$ when $Q=(1)$ because $(1)$ is $L$-trivial.
Notice that any resolution of the diagram $T(p,q)$ results in multiple applications of the Reidemeister I move, so we have that $T(p,q)$ is $K$-trivial.
Hence, both $(1)$ and $T(p,q)$ have a trivializing number of 0 as diagrams of type $L$ and type $K$, respectively. \\

$(\Rightarrow)$
Let $P=(x_1, x_2, \cdots, x_k)$ be a pretzel projection and $Q$ be a pseudodiagram obtained from $P$ by resolving some of the precrossings of $P$ and applying any possible Reidemeister moves.
Then $Q=(x_1', x_2', \cdots, x_k')$ where $x_i ' \geq 0$ for all $1 \leq i \leq k$.
Suppose that $tr(Q)=0$.
By Lemma~\ref{sk}, we know that if one of the stacks is knotted, then the entire diagram is knotted.
Since $tr(Q)=0$ and two precrossings can already provide a nontrivial diagram (by resolving the precrossings with the same type of crossing), we must have that each stack contains one precrossing, or every stack contains a single precrossing except for one stack which does not contain any precrossings depending on whether or not $P$ contains a stack with an even number of precrossings.

If $x_i'=1$ for all $i$, then we have a horizontal stack of precrossings.
If $i=1$, then we have the diagram $(1)$ which is $L$-trivial, so $Q=(1)$ is a possible diagram with $tr(Q)=0$.
However, if $i \geq 2$, we can resolve all of the precrossings using the same type of crossing and have no instances where the Reidemeister II move can be applied.
This resolution is a knotted diagram, contradicting the fact that $tr(Q)=0$.
As a result, the diagram $(1)$ is the only possible trivial diagram in this case.

If $x_j'=0$ for some $1 \leq j \leq k$ while $x_i'=1$ for all $i \neq j$, we have the diagram $T(p,q)$ for some $p, q \geq 0$, which we know to be $K$-trivial.

Therefore, the result holds.
\end{proof}


\subsection{Single Stack of Precrossings}

In this section, we calculate the trivializing and knotting numbers of pretzel projections that contain only a single stack of precrossings. 
In an effort to simplify our explanations when calculating these numbers, we introduce the following notation.
We assign $R_+$ to the crossing where the overstrand has a positive slope, and $R_-$ to the crossing where the overstrand has a negative slope. 

\begin{figure}[ht]
\[
\raisebox{20pt}{$R_{+}=$}\includegraphics[height=1.5cm]{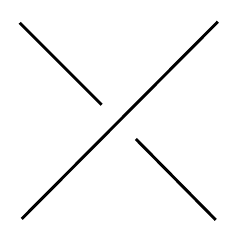}
\quad \quad
\raisebox{20pt}{$R_{-}=$}\includegraphics[height=1.5cm]{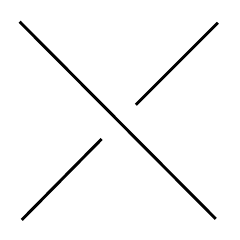} 
\]
\caption{$R_{+}$ and $R_{-}$}
\end{figure}

Let $\N_0=\N \cup \{0\}$.

\begin{prop}\label{1tr}
Given a pretzel projection $P=(x)$, where $x \in \N_0$, the trivializing number of $P$ is given by

\vspace{-.25cm}

\[
tr(P)=
2 \left\lfloor \frac{x}{2} \right\rfloor.
\]
\end{prop}

\begin{proof}
First, we show that $tr(P) \leq 2 \left  \lfloor \frac{x}{2} \right \rfloor$. 
If we resolve $\left \lfloor \frac{x}{2} \right \rfloor$ precrossings with $R_{+}$ and $\left \lfloor \frac{x}{2} \right \rfloor$ precrossings with $R_{-}$, we have created a situation where we can implement $\left \lfloor \frac{x}{2} \right \rfloor$ many Reidemeister II moves. 
If $x$ is even, this process results in diagram $(0)$, which is $K$-trivial and if $x$ is odd, we obtain the diagram $(1)$, which is $L$-trivial. 
Regardless of the parity of $x$, our diagram is trivial. 

To prove that $tr(P)=2 \left \lfloor \frac{x}{2} \right \rfloor$, we consider the next possible trivializing number. 
Since the trivializing number is even by Lemma~\ref{tr1}, the next value to consider is $2 \left \lfloor \frac{x}{2} \right \rfloor-2$.
To create the most trivial diagram as possible, we can perform $\left \lfloor \frac{x}{2} \right \rfloor -1$ Reidemeister II moves by resolving $ \left \lfloor \frac{x}{2} \right \rfloor-1$ precrossings with $R_{+}$ and $ \left \lfloor \frac{x}{2} \right \rfloor-1$ precrossings with $R_{-}$. 
Depending on the parity of $x$, we obtain diagram $(2)$ or $(3)$ after resolving the precrossings. 
Since we can resolve all of the remaining precrossings with $R_+$ (or $R_-$) and obtain a knotted diagram, we have that $tr(P)=2 \left \lfloor \frac{x}{2} \right \rfloor$.
\end{proof}

\begin{prop}\label{1kn}
Given a pretzel projection $P=(x)$, where $x \in \N_0$, the knotting number is given by

\vspace{-.25cm}

\[
kn(P)=
\begin{cases} 
\infty & \text{ if $x=0,1$} \\ 
\displaystyle \left \lceil \frac{x}{2} \right \rceil +1 & \text{ if $x>1$. } 
\end{cases}
\]
\end{prop}

\begin{proof}
First consider $x=0$ and $x=1$.
Diagrams $(0)$ and $(1)$ are $K$-trivial and $L$-trivial, respectively. 
Since their knotting numbers do not exist, we label them as infinity. 

Consider the case when $x>1$. 
We begin by showing that $kn(P) \leq \left \lceil \frac{x}{2} \right \rceil +1$.
If we resolve $\left \lceil \frac{x}{2} \right \rceil +1$ precrossings with $R_{+}$, even if we maximized the number of Reidemeister II moves by resolving $\left \lfloor \frac{x}{2} \right \rfloor-1$ precrossings with $R_{-}$, we can apply at most $\left \lfloor \frac{x}{2} \right \rfloor -1$ Reidemeister II moves. 
Depending on the parity of $x$, we obtain either 2 or 3 $R_{+}$ crossings after implementing the Reidemeister II moves, which correspond to nontrivial diagrams. 
Therefore, our diagram is knotted. 

To show equality, suppose that we resolved $\left \lceil \frac{x}{2} \right \rceil$ precrossings with $R_{+}$. 
To create the most Reidemeister II moves as possible, we can resolve $\left \lfloor \frac{x}{2} \right \rfloor$ precrossings with $R_{-}$ and apply $\left \lfloor \frac{x}{2} \right \rfloor$ Reidemeister II moves. 
After we apply the Reidemeister II moves, we obtain either diagrams $(0)$ or $(1)$ depending on the parity of $x$, which we know to be trivial. 
Since we obtained a trivial diagram, we have that $kn(P)=\left \lceil \frac{x}{2} \right \rceil +1$.
\end{proof}


\subsection{Finite Number of Stacks of Precrossings}

In this section, we begin calculating the trivializing and knotting numbers of pretzel projections that contain a finite number of stacks.

\begin{thm}\label{BigTr}
Given a pretzel projection $P=(x_1, x_2, \cdots, x_m, y_1, y_2, \cdots, y_n)$ where $m,n \geq 0$, $x_i=1$ for all $1 \leq i \leq m$, and $y_j \geq 2$ for all $1 \leq j \leq n$, the trivializing number of $P$ is given by 

\vspace{-.25cm}

\[
\small{
tr(P)=
\begin{cases}
\displaystyle 2 \left\lfloor \frac{m}{2} \right\rfloor & \text{if $n=0$} \\
y_j+\displaystyle \sum_{i=1, i \neq j}^n \left(y_i-1\right)& \text{if $n \geq 1$ and one $y_j$ is even } \\
\displaystyle 2 \left\lfloor \frac{m+n}{2} \right\rfloor+\displaystyle \sum_{j=1}^n (y_j-1) & \text{if $n \geq 1$ and all $y_i$ are odd.} 
\end{cases}
}
\]
\end{thm}

\begin{proof}
Consider the case when $n=0$.
By our convention, $(x_1, x_2, \cdots, x_m)=(m)$ since $x_i=1$ for all $1 \leq i \leq m$. 
By Proposition~\ref{1tr}, the trivializing number for diagram $(m)$ is $2 \left \lfloor \frac{m}{2} \right \rfloor$.

Next, suppose that $n \geq 1$ and one $y_j$ is even for some $1 \leq j \leq n$. 
To ensure the diagram is trivial, we must resolve pairs of precrossings in each stack of $y_i$ precrossings to create as many instances of the Reidemeister II move as possible, because two precrossings in a stack can be resolved to create a nontrivial diagram.
After implementing as many instances of the Reidemeister II move as possible, we have eliminated all of the precrossings in the $y_j$ stack,
all but one of the precrossings in the stacks $y_i$, $i \neq j$, and reduced the diagram to $T(p,q)$ for some $p,q \geq 0$. 
Since the diagram $T(p,q)$ is $K$-trivial,  $tr(P)= y_j+\sum_{i=1, i \neq j}^n \left(y_i-1\right)$.

Assume that $n \geq 1$ and all of the values of $y_j$ are odd.
Similar to the previous case, we must resolve pairs of precrossings to create as many instances of the Reidemeister II move as possible to ensure the diagram is trivial. 
After we resolve the precrossings and apply the Reidemeister II moves, we have reduced the diagram to $(m+n)$. 
By Proposition~\ref{1tr}, the trivializing number of $(m+n)$ is $2 \left \lfloor \frac{m+n}{2} \right \rfloor$.
The result follows.
\end{proof}

\begin{thm} \label{BigKn}
Given a pretzel projection $P=(x_1, x_2, \cdots, x_m, y_1, y_2, \cdots, y_n)$ where $m,n \geq 0$, $x_i=1$ for all $1 \leq i \leq m$, and $y_j \geq 2$ for all $1 \leq j \leq n$, the knotting number of $P$ is given by

\vspace{-.25cm}

\small{
\[
kn(P)=
\begin{cases} 
\infty & \text{if $n=0$ and $m=0, \,1$} \\
\displaystyle \left\lceil \frac{m}{2} \right \rceil +1 & \text{if $n=0$ and $m > 1$} \\
\displaystyle \min \left\{ \left\lceil \frac{y_1}{2} \right\rceil+1, \cdots, \left\lceil \frac{y_n}{2} \right\rceil +1\right\} & \text{if $n \geq 1$, all $y_j$ are odd,} \\
& \text{and $m \leq n+1$} \\
& \text{or $n \geq 1$ and one $y_j$} \\
& \text{is even} \\
\displaystyle \min \left\{ \left\lceil \frac{y_1}{2} \right\rceil+1, \cdots, \left\lceil \frac{y_n}{2} \right\rceil +1, \left\lceil \frac{m+n}{2} \right \rceil +1 \right\} & \text{if $n \geq  1$, all $y_j$ are odd,} \\
& \text{and $m>n+1$.}
\end{cases}
\]
}
\end{thm}

\begin{proof}
Consider the case when $n=0$ and $m=0$ or $1$. 
The two possible diagrams correspond to $(0)$ and $(1)$ which are trivial, so the knotting number is infinity. 

Suppose that $n=0$ and $m > 1$. 
Since $(x_1, x_2, \cdots, x_m)=(m)$, the knotting number for diagram $(m)$ is $\left \lceil \frac{m}{2} \right \rceil +1$ (by Proposition~\ref{1kn}).

Suppose that $n \geq 1$, all $y_i$ are odd, and $m \leq n+1$.
We want to show that since $m \leq n+1$, the $m$ initial precrossings in our diagram do not affect our knotting number. 
Consider the situation where we resolve all of the $m$ precrossings with $R_{+}$, in an attempt to make the diagram as knotted as possible. 
We can resolve the remaining precrossings in the stacks of $P$ by creating as many Reidemeister II moves as possible in each stack, and resolving the remaining precrossing in each stack with $R_{-}$.
After applying all of the possible Reidemeister II moves, our diagram has been reduced to a diagram containing $m-n$ crossings of type $R_{+}$. 
Notice that $m-n \leq (n+1)-n=1$.
Therefore, we have at most a single instance of $R_{+}$ depending on the parity of $m$ and $n$.
Regardless, we have created a trivial diagram. 
As a result, the $m$ precrossings do not affect the knotting number.
By Lemma~\ref{sk}, we know that if one of the stacks of $y_j$ precrossings is knotted, the entire diagram is knotted, and Proposition~\ref{1kn} states that the knotting number of a stack of $y_j$ precrossings is $\left\lceil \frac{y_j}{2} \right\rceil+1$.  
Furthermore, we require the knotting number to be the minimum number of precrossings to be resolved.  
The result follows.

Next consider the case when $n \geq 1$ and one $y_j$ is even.
Similar to the previous case, we need to show that the $m$ precrossings do not affect the knotting number. 
We can resolve the precrossings in the $n$ stacks which have 2 or more precrossings to create as many instances of the Reidemeister II move as possible. 
By doing so, we separate the two strands that interact in the $j$th stack and transform our diagram into the diagram $T(p,q)$ for some $p,q \geq 0$. 
Therefore, regardless of how we resolve the $m$ initial precrossings, we will always have a trivial diagram. 
As a result, the only way to create a knotted diagram is to ensure that one of the stacks of $y_j$ precrossings is knotted and the result follows.

Lastly, suppose that $n \geq 1$, all $y_i$ are odd, and $m>n+1$. 
We claim that because $m>n+1$, the $m$ initial precrossings could affect the knotting number.
Note that $m>n+1$ implies that $m \geq n+2$ because $m$ is an integer. 
Furthermore, $\left \lceil \frac{m+n}{2} \right \rceil +1 \leq m$ because $n \leq  m-2$.
This statement is verified in the calculations below. 
\begin{eqnarray*}
\left \lceil \frac{m+n}{2} \right \rceil +1 \leq \left \lceil \frac{m+(m-2)}{2} \right \rceil +1 =\left \lceil m-1\right \rceil +1=m
\end{eqnarray*}
Suppose we resolve $\left \lceil \frac{m+n}{2} \right \rceil +1$ of the $m$ initial precrossings with $R_{+}$.
Even if we resolve the remaining precrossings of $P$ by creating as many Reidemeister II moves in each of the $y_j$ stacks and resolve the remaining single precrossing in each stack with $R_{-}$, we have at least 2 $R_{+}$ crossings remaining as seen in the calculations below:
\begin{eqnarray*}
\left \lceil \frac{m+n}{2} \right \rceil +1 - \left(m+n - \left(\left \lceil \frac{m+n}{2} \right \rceil +1 \right)\right) &=& 2\left \lceil \frac{m+n}{2} \right \rceil +2-m-n  \\
& \geq & 2\left(\frac{m+n}{2}\right)+2-m-n \\
&=& 2.
\end{eqnarray*}
Since our diagram contains at least two crossings with the same type of crossing, the diagram is knotted. 
Thus, if we resolve $\left \lceil \frac{m+n}{2} \right \rceil +1$ many precrossings, we can ensure that $P$ is always knotted, regardless of how we resolve the remaining precrossings.
Because the knotting number is the minimum number of precrossings that need to be resolved to ensure that the diagram is knotted, the knotting number for this case will be the minimum of the knotting numbers for each of the $y_j$ stacks and  $ \left \lceil \frac{m+n}{2} \right \rceil +1$. 
\end{proof}


\section{Pretzel Projections with given trivializing and knotting numbers}

In previous sections, we were given a pretzel projection and we calculated its trivializing and knotting numbers.
In this section, we switch our focus by starting with a number and creating a pretzel projection with that number as its trivializing or knotting number. 

\begin{prop}
For any non-negative even number $t$, there exists a pretzel projection $P$ with $tr(P)=t$.
\end{prop}

\begin{proof}
By Proposition~\ref{1tr}, we know that the pretzel projection $P=(t)$ has $tr(P)=t$, since $2 \left \lfloor \frac{t}{2} \right \rfloor =t$.
\end{proof}

\begin{prop}\label{TK}
For any $k \in \N$, where $k \geq 2$, there exists a pretzel projection $P$ with $kn(P)=k$.
\end{prop}

\begin{proof}
By Proposition~\ref{1kn}, we have that the pretzel projection $P=(2(k-1))$ has knotting number equal to $k$, since $ \left \lceil \frac{2(k-1)}{2} \right \rceil+1=k-1+1=k$.
\end{proof}

The next result establishes the relationship between the trivializing and knotting numbers.
Hanaki showed in~\cite{H} that the trivializing and knotting numbers for knot projections are independent of each other (see~\cite[Proposition 1.9]{H}).
However, in our case of projections of 2-bouquets, we find that the trivializing number depends on the knotting number of the projection.

\begin{thm}
For any $k \in \N$ where $k \geq 2$, there exists a pretzel projection $P=(x_1, x_2, \cdots, x_m, y_1, y_2, \cdots, y_n)$, where not both $m$ and $n$ are zero, $x_i=1$ for all $1 \leq i \leq m$, $y_j \geq 2$ for all $1 \leq j \leq n$, with $kn(P)=k$, and the following trivializing numbers for some $l \in \N_0$:  

\vspace{-.25cm}

\[
tr(P)=
\begin{cases}
2(k-1) \text{ or } 2(k-2) & \text{if $n=0$} \\
2n(k-2)+2+2l & \text{if $n \geq 1$ and one $y_j$ is even} \\
\displaystyle 2\left \lfloor \frac{m+n}{2} \right \rfloor +2n(k-2) +2l & \text{if $n \geq 1$, all $y_j$ are odd, } \\
& \text{and $m \leq n+1$} \\
2(k-1)(n+1)+2l \text{ or } & \text{if $n \geq 1$, all $y_j$ are odd, } \\
\displaystyle 2\left \lfloor \frac{m+n}{2} \right \rfloor +2n(k-2) +2l  & \text{and $m > n+1$}
\end{cases}
\]
\end{thm}

\begin{proof}
Consider the case when $n=0$ and $m \geq 1$.
Then $P=(m)$.
By Proposition~\ref{1kn}, we know that $kn(P)=\left \lceil \frac{m}{2} \right \rceil+1$. 
To ensure that $kn(P)=k$, $m$ must be equal to the following expressions: $2(k-1)$ or $2(k-1)-1$. 
By Proposition~\ref{1tr}, we obtain $tr(P)=2(k-1)$ and $tr(P)=2(k-2)$ for the values of $m=2(k-1)$ and $m=2(k-1)-1$, respectively. 

Suppose next that $n \geq 1$ and one $y_j$ is even. 
By Theorem~\ref{BigKn}, $kn(P)=\left \lceil \frac{y_i}{2} \right \rceil +1$ for some $1 \leq i \leq n$. 
To force $kn(P)=k$, we must have that either $y_i=2(k-1)$ or $y_i=2(k-1)-1$.
Consider first the case $y_i=2(k-1)-1$.
We want to create a pretzel projection $P$ that contains the fewest number of precrossings that still has $kn(P)=k$. 
Without loss of generality, let the $n$th stack contain an even number of precrossings.
To maintain the smallest number of precrossings as possible, we will set $y_n=2(k-1)$. 
All other values for the $y_i$ stacks can be given the value of $2(k-1)-1$ and still maintain $kn(P)=k$. 
Therefore, the diagram with the fewest number of precrossings that has $kn(P)=k$ will be of the form $P=(x_1, x_2, \cdots, x_m, 2(k-1)-1, \cdots, 2(k-1)-1, 2(k-1))$.
To obtain all possible trivializing numbers, we can arbitrarily add $2l$ to one of the stacks where $l \in \N_0$.
For example, let $P=(x_1, x_2, \cdots, x_m, 2(k-1)-1, \cdots, 2(k-1)-1, 2(k-1)+2l)$. 
By Theorem~\ref{BigTr}, 
\begin{eqnarray*}
tr(P) &=& 2(k-1)+2l+ \sum_{i=1, i \neq j}^n (2(k-1)-1-1) \\
&=& 2(k-1)+2l+(n-1)(2(k-2)) \\
&=& 2(k-1+(n-1)(k-2))+2l \\
&=& 2(n(k-2)+1)+2l \\
&=& 2n(k-2)+2+2l.
\end{eqnarray*}
Using similar reasoning for the case $y_i=2(k-1)$, we have that $y_i$ is the single stack with an even number of precrossings and to obtain the smallest number of precrossings in the diagram, all of the other stacks will contain $2(k-1)+1$ precrossings. 
We will also arbitrarily add $2l$ to the $n$th stack in the diagram where $l \in \N_0$.
Therefore, $P$ is of the form $(x_1, x_2, \cdots, x_m, 2(k-1)+1, \cdots, 2(k-1)+1, 2(k-1)+2l)$.
By Theorem~\ref{BigTr},
\begin{eqnarray*}
tr(P)&=& 2(k-1)+2l+ \sum_{i=1, i \neq j}^n (2(k-1)+1-1) \\
&=& 2(k-1)+2l+(n-1)(2(k-1)) \\
&=& 2n(k-1)+2l.
\end{eqnarray*}
Notice that $2n(k-2)+2 < 2n(k-1)$ since $n \geq 1$ and $k \geq 2$.
Therefore, we will use the first construction where the trivializing number is given by $2n(k-2)+2+2l$ to obtain as many trivializing numbers as possible. 

Consider the case when $n \geq 1$, all $y_j$ are odd, and $m \leq n+1$. 
Similar to the previous case, given that $kn(P)=\left \lceil \frac{y_i}{2} \right \rceil+1$ for some $1 \leq i \leq n$, we require $y_i=2(k-1)$ or $y_i=2(k-1)-1$.
Since we are in the case where all of the $y_j$ stacks are odd, we must have that $y_i=2(k-1)-1$. 
Therefore, the smallest possible diagram with $kn(P)=k$ will be of the form $P=(x_1, x_2, \cdots, x_m, 2(k-1)-1, \cdots, 2(k-1)-1, 2(k-1)-1)$. 
To guarantee that we obtain all possible trivializing numbers we will arbitrarily add $2l$ to the $n$th stack where $l \in \N_0$. 
Thus, $P=(x_1, x_2, \cdots, x_m, 2(k-1)-1, \cdots, 2(k-1)-1, 2(k-1)-1+2l)$. 
By Theorem~\ref{BigTr}, 
\begin{eqnarray*}
tr(P) &=& 2 \left \lfloor \frac{m+n}{2} \right \rfloor + \sum_{j=1}^n (2(k-1)-1-1) +2l \\
&=&  2 \left \lfloor \frac{m+n}{2} \right \rfloor+2n(k-2)+2l.
\end{eqnarray*}

Lastly, suppose that $n \geq 1$, all $y_j$ are odd, and $m>n+1$.
According to Theorem~\ref{BigKn}, $kn(P)=\min\{\left \lceil \frac{y_1}{2} \right \rceil+1, \cdots, \left \lceil \frac{y_n}{2} \right \rceil+1, \left \lceil \frac{m+n}{2} \right \rceil+1\}$.
If $kn(P)=\left \lceil \frac{y_i}{2} \right \rceil+1$ for some $ 1 \leq i \leq n$, by our work in the previous case, we know that $tr(P)=2 \left \lfloor \frac{m+n}{2} \right \rfloor+2n(k-2)+2l$. 
If $kn(P)=\left \lceil \frac{m+n}{2} \right \rceil+1$, given that $kn(P)=k$, we must have that $m+n=2(k-1)$ or $m+n=2(k-1)-1$. 
If $m+n=2(k-1)$, the projection with the fewest number of precrossings is $P=(x_1, x_2, \cdots, x_m, 2(k-1)+1, \cdots, 2(k-1)+1)$.
To obtain all possible trivializing numbers, we can add $2l$ to any of the $y_i$ stacks for some $l \in \N_0$. 
Therefore, we are considering $P=(x_1, x_2, \cdots, x_m, 2(k-1)+1, \cdots, 2(k-1)+1, 2(k-1)+1+2l)$. 
By Theorem~\ref{BigTr}, 
\begin{eqnarray*}
tr(P) &=& 2 \left \lfloor \frac{2(k-1)}{2} \right \rfloor+\sum_{j=1}^n (2(k-1)+1-1)+2l \\
&=& 2(k-1)+2n(k-1)+2l \\
&=& 2(k-1)(n+1)+2l.
\end{eqnarray*}
Using similar reasoning with $m+n=2(k-1)-1$, we obtain that the possible diagrams with $kn(P)=k$ are $P=(x_1, x_2, \cdots, x_m, 2(k-1)+1, \cdots, 2(k-1)+1, 2(k-1)+1+2l)$ for some $l \in \N_0$. 
Then $tr(P)=2(k-1)(n+1)+2l$ as well. 
\end{proof}

For the remaining results in this section, we introduce an additional definition.
Given a pretzel projection $P$, let \textbf{\textit{p(P)}} be the number of precrossings in $P$.

\begin{rem}
In ~\cite[Theorem 1.12]{H}, it was shown that there are infinitely many possible knot projections $P$ with $tr(P)=p(P)$ where $p(P)$ is the number of precrossings in the knot projection. 
In the case of projections of 2-bouquets, we too have infinitely many projections with this property. 
\end{rem}

\begin{prop}\label{t=p}
Consider a pretzel projection $P=(x_1,x_2, \cdots, x_m, y_1, y_2, \newline \cdots, y_n)$ where $m,n \geq 0$, $x_i=1$ for all $1 \leq i \leq m$, and $y_j \geq 2$ for all $1 \leq j \leq n$. 
Then $tr(P)=p(P)$ if and only if $n=0$ and $m$ is even or $n \geq 1$, $m+n$ is even, and $y_j$ are odd for all $1 \leq j \leq n$. 
\end{prop}

\begin{proof}
$(\Leftarrow)$ We have this implication by Proposition~\ref{1tr} and Theorem~\ref{BigTr}.

$(\Rightarrow)$
Consider the case of $n=0$.
Recall that $P=(x_1,x_2, \cdots, x_m)=(m)$ if $x_i=1$ for all $1 \leq i \leq m$.
By Proposition~\ref{1tr}, $tr(P)=2 \left \lfloor \frac{m}{2} \right \rfloor$. 
Furthermore, $tr(P)=2 \left \lfloor \frac{m}{2} \right \rfloor=m=p(P)$ only when $m$ is even. 
Thus, we have the first result.

Suppose next that $n \geq 1$ and one $y_j$ is even. 
By Theorem~\ref{BigTr}, $tr(P)=y_j+\sum_{i=1, i \neq j}^n (y_i-1)$. 
Note that 
\[
tr(P)=y_j+\sum_{i=1, i \neq j}^n (y_i-1)< \sum_{i=1}^n y_i \leq  m+ \sum_{i=1}^n y_i=p(P).
\]
Therefore, it is not possible for $tr(P)=p(P)$ in this case. 

Lastly, suppose that $n \geq 1$ and all $y_j$ are odd. 
By Theorem~\ref{BigTr}, $tr(P)=2\left \lfloor \frac{m+n}{2} \right \rfloor + \sum_{j=1}^n (y_j-1)$. 
For equality between the trivializing number and the number of precrossings to be achieved, we require the following:
\[
tr(P)=2\left \lfloor \frac{m+n}{2} \right \rfloor + \sum_{j=1}^n (y_j-1)=m+n+\sum_{j=1}^n (y_j-1)=m+\sum_{j=1}^n y_j=p(P)
\]
Thus, we require that $2\left \lfloor \frac{m+n}{2} \right \rfloor=m+n$. 
This is only true when $m+n$ is even.
\end{proof}

\begin{prop}\label{k=p}
Consider a pretzel projection $P=(x_1, x_2, \cdots, x_m, y_1, y_2, \newline \cdots, y_n)$, where $m,n \geq 0$, $x_i=1$ for all $1 \leq i \leq m$, $y_j \geq 2$ for all $1 \leq j \leq n$. Then $kn(P)=p(P)$ if and only if $P=(2)=(1,1)$ or $P=(3)=(1,1,1)$. 
\end{prop}

\begin{proof}
$(\Leftarrow)$ This implication is a consequence of Proposition~\ref{1kn}. 

$(\Rightarrow)$ 
Suppose first that $n=0$. 
Notice that $m=0$ and $m=1$ have $kn(P)=\infty$ and $p(P) <\infty$, so this case does not result in any possible projections with the required equality. 
If $m>1$, then $P=(x_1, x_2, \cdots, x_m)=(m)$ and by Proposition~\ref{1kn}, $kn(P)=\left \lceil \frac{m}{2} \right \rceil+1$.
Therefore, we require that 
\[
kn(P)=\left \lceil \frac{m}{2} \right \rceil+1=m=p(P).
\]
Notice that this equality is true only when $m=2$ or $m=3$. 
Otherwise, $\left \lceil \frac{m}{2} \right \rceil+1 < m$ for all $m > 3$.
Hence for the case of $n=0$, the only possible pretzel projections that have $kn(P)=p(P)$ are $P=(1,1)=(2)$ and $P=(1,1,1)=(3)$. 

Consider the case when $n \geq 1$, $m \leq n+1$, and all $y_j$ are odd or $n \geq 1$ and one $y_j$ is even. 
By Theorem~\ref{BigKn}, $kn(P)=\left \lceil \frac{y_i}{2} \right \rceil +1$ for some $1 \leq i \leq n$.
Consequently, we require that 
\[
kn(P)=\left \lceil \frac{y_i}{2} \right \rceil +1=m+\sum_{j=1}^n y_j=p(P).
\]
If $n \geq 2$, we have that 
\[
kn(P)=\left \lceil \frac{y_i}{2} \right \rceil +1 <  \sum_{j=1}^n y_j \leq m+\sum_{j=1}^n y_j=p(P).
\]
So, if $n \geq 2$, it is not possible to have $kn(P)=p(P)$. 
Therefore, we require $n=1$ and $\left \lceil \frac{y_i}{2} \right \rceil +1=m+y_i$.
If $m=0$, $y_i=2$ or $3$ makes the statement true, but no other values are possible as seen in the previous case. 
If $m=1$, $y_i=0$ or $1$ makes the statement true, but $y_i \geq 2$ for all values of $1 \leq i \leq n$, so we must discard this case. 
If $m>1$, $m+y_i > 1+y_i \geq 1+\left \lceil \frac{y_i}{2} \right \rceil$, so it is not possible to obtain $kn(P)=p(P)$. 
Therefore, the only possible pretzel projections are when $m=0$ and $y_i=2$ or $3$ which corresponds to $P=(2)$ and $P=(3)$. 

Lastly, consider the case of $n \geq 1$, $m>n+1$, and all $y_j$ are odd. 
In this case, $kn(P)=\left \lceil \frac{y_i}{2} \right \rceil +1$ for some $1 \leq i \leq n$ or $kn(P)=\left \lceil \frac{m+n}{2} \right \rceil+1$ by Theorem~\ref{BigKn}.
Assume that $kn(P)=\left \lceil \frac{y_i}{2} \right \rceil +1$.
Similar to the previous case, we require 
\[
kn(P)=\left \lceil \frac{y_i}{2} \right \rceil +1=m+\sum_{j=1}^n y_j=p(P).
\]
Recall that when we considered $n \geq 2$, we could not achieve equality, so this case is discarded. 
If $n=1$, then $m>2$ since $m>n+1$. 
However, we have already discovered that with $n=1$ and $m>1$, there is no projection that has $kn(P)=p(P)$. 
Therefore, if $kn(P)=\left \lceil \frac{y_i}{2} \right \rceil +1$ with $n \geq 1$ and $m>n+1$, there are no possible projections with $kn(P)=p(P)$. 
Suppose now that $kn(P)=\left \lceil \frac{m+n}{2} \right \rceil+1$.
We require
\[
kn(P)=\left \lceil \frac{m+n}{2} \right \rceil+1=m+\sum_{j=1}^n y_j=p(P).
\]
However, notice that since $n \geq 1$ and $m>n+1$, 
\[
kn(P)=\left \lceil \frac{m+n}{2} \right \rceil+1 < m+n<m+n+\sum_{j=1}^n(y_j-1)=p(P)
\]
Therefore, the case when $n \geq 1$ and $m>n+1$ does not produce any pretzel projections with $kn(P)=p(P)$. 

The result follows when considering all of the cases.
\end{proof}

\begin{rem}
We want to bring attention to the fact that according to~\cite[Theorem 1.13]{H}, there are three projections of knots that have the property that $kn(P)=p(P)$. 
Proposition~\ref{k=p} shows that in our case, we have only two pretzel projections of 2-bouquets that have this property. 
\end{rem}


\section{Prime 2-Bouquets}

\subsection{Trivializing and Knotting Numbers of Prime 2-Bouquets}

So far, we have calculated the trivializing and knotting numbers of projections of 2-bouquets based on a particular construction, namely our pretzel projections. 
In this section, we want to find the trivializing and knotting numbers for projections of 2-bouquets based on the number of precrossings in the projection.
We accomplish this task by considering a list of prime 2-bouquets given by Oyamaguchi in her doctoral dissertation~\cite{O}. 
\textbf{Prime 2-bouquets} are 2-bouquets such that when an arbitrary 2-sphere in $\mathbb{R}^3$ intersects the graph at two points, it divides the graph into a trivial arc and the remaining graph. 
Oyamaguchi lists all of the prime 2-bouquets of type $K$ and $L$ up to six crossings by constructing them from prime 2-string tangles. We provide Oyamaguchi's complete list in the Appendix.

We borrow the notation for prime 2-bouquets introduced in~\cite{O}, and denote by $\overline{D}$ the flat version of a prime 2-bouquet diagram $D$, where $D$ is any diagram in the list given in~\cite{O}. That is, $\overline{D}$ is a projection of a prime 2-bouquet. We have computed the trivializing and knotting numbers of projections $\overline{D}$  of all of the prime 2-bouquets of type $K$ (and type $L$, respectively) up to six crossings. Our results are listed in  Table~\ref{table:1} (and Table~\ref{table:2}, respectively).
We provide two of the calculations in the following examples.

\begin{table}[ht]
\caption{Trivializing and Knotting Numbers of 2-Bouquets of Type $K$}
\begin{center}
\begin{tabular}{ c|c|c  } 
Prime 2-Bouquet Graph & Trivializing Number & Knotting Number \\
\hline
$\overline{0_1^k}$ & 0 & $\infty$ \\
\hline 
 $\overline{2_1^k}$ & 2 & 2 \\
\hline
 $\overline{3_1^k}$ & 2 & 2 \\
\hline
 $\overline{4_1^k}$ & 4 & 3 \\ 
\hline
 $\overline{4_2^k}$ & 4 & 3 \\ 
\hline
 $\overline{4_3^k}$ & 2 & 2 \\ 
\hline
 $\overline{5_1^k}$ & 4 & 3 \\ 
\hline
 $\overline{5_2^k}$ & 2 & 2 \\ 
\hline
 $\overline{5_3^k}$ & 4 & 3 \\ 
\hline
 $\overline{5_4^k}$ & 4 & 2 \\ 
\hline
 $\overline{5_5^k}$ & 4 & 2 \\ 
\hline
 $\overline{5_6^k}$ & 4 & 2 \\ 
\hline
 $\overline{5_7^k}$ & 4 & 2 \\ 
\hline
 $\overline{5_8^k}$ & 4 & 2 \\ 
\hline
 $\overline{6_1^k}$ & 6 & 4 \\ 
\hline
 $\overline{6_2^k}$ & 6 & 4 \\ 
\hline
 $\overline{6_3^k}$ & 2 & 2 \\ 
\hline
 $\overline{6_4^k}$ & 4 & 3 \\ 
\hline
 $\overline{6_5^k}$ & 6 & 3 \\ 
\hline
 $\overline{6_6^k}$ & 4 & 3 \\ 
\hline
 $\overline{6_7^k}$ & 6 & 3 \\ 
\hline
 $\overline{6_8^k}$ & 6 & 3 \\ 
\hline
 $\overline{6_9^k}$ & 4 & 2 \\ 
\hline
 $\overline{6_{10}^k}$ & 4 & 3\\ 
\hline
 $\overline{6_{11}^k}$ & 4& 3\\ 
\hline
 $\overline{6_{12}^k}$ & 4 & 2 \\ 
\hline
 $\overline{6_{13}^k}$ & 4 & 3 \\ 
\hline
 $\overline{6_{14}^k}$ & 4 & 2 \\ 
\hline 
$\overline{6_{15}^k}$ & 4& 3\\ 
\hline
$\overline{6_{16}^k}$ & 4 & 2 \\ 
\hline
$\overline{6_{17}^k}$ & 4 & 2 \\ 
\hline
$\overline{6_{18}^k}$ & 4 & 2 \\ 
\hline
$\overline{6_{19}^k}$ & 4 & 2 \\ 
\hline
\end{tabular}
\end{center}
\label{table:1}
\end{table}

\begin{table}[ht]
\caption{Trivializing and Knotting Numbers of 2-Bouquets of Type $L$}
\begin{center}
\begin{tabular}{ c|c| c  } 
Prime 2-Bouquet Graph & Trivializing Number & Knotting Number \\
\hline
 $\overline{1_1^l}$ & 0 & $\infty$ \\ 
\hline 
 $\overline{3_1^l}$ & 2 & 3 \\ 
\hline
 $\overline{4_1^l}$ & 2 & 3 \\ 
\hline
 $\overline{5_1^l}$ & 4 & 4 \\ 
\hline
 $\overline{5_2^l}$ & 4 & 3 \\ 
\hline
 $\overline{5_3^l}$ & 2 & 2 \\ 
\hline
 $\overline{6_1^1}$ & 4 & 3 \\ 
\hline
 $\overline{6_2^l}$ & 4 & 3 \\ 
\hline
 $\overline{6_3^l}$ & 4 & 3 \\ 
\hline
 $\overline{6_4^l}$ & 4 & 3 \\ 
\hline
 $\overline{6_5^l}$ & 4 & 3 \\ 
\hline
 $\overline{6_6^l}$ & 2 & 2 \\ 
\hline
 $\overline{6_7^l}$ & 4 & 2 \\ 
\hline
 $\overline{6_8^l}$ & 4 & 2 \\ 
\hline
 $\overline{6_9^l}$ & 4 & 3 \\ 
\hline
 $\overline{6_{10}^l}$ & 4 & 2 \\ 
\hline
 $\overline{6_{11}^l}$ & 4 & 2 \\ 
\hline
 $\overline{6_{12}^l}$ & 4 & 3 \\ 
\hline
\end{tabular}
\end{center}
\label{table:2}
\end{table}

\begin{ex}
Consider the projection $\overline{5_1^k}$ depicted in Figure~\ref{fig: 51k}, with labels 1 and 2 on its petals.

\begin{figure}[ht]
\[
\includegraphics[height=2.5cm]{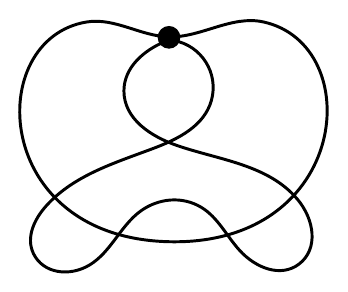}
\put(-47, 42){\fontsize{10}{11}2}
\put(-90, 40){\fontsize{10}{11}1}
\]
\caption{The Projection $\overline{5_1^k}$}
\label{fig: 51k}
\end{figure}

If we travel along petal 1 and resolve the precrossings so that the strand we travel along alternates between becoming an overstrand and an understrand, we create a nontrivial diagram, because the petals are linked together.
This implies that $\overline{5_1^k}$ is nontrivial, $tr(\overline{5_1^k}) \neq 0$ and $kn(\overline{5_1^k}) \leq 4$. 
Moreover, if we travel along petal 1 and resolve the precrossings so that petal 1 is always the overstrand, we obtain a trivial diagram. 
Thus, we have that $tr(\overline{5_1^k}) \leq 4$. 
Therefore, we can conclude that only the precrossings that interact with petal 1 determine if the diagram is trivial or knotted.

Suppose that we resolve any three of the precrossings formed between the two petals so that petal 1 becomes the overstrand.
Even if we resolved the remaining precrossing to have petal 1 become an understrand and employ a Reidemeister II move, the diagram would be nontrivial as the two petals would still be linked together. 
This statement proves that $kn(\overline{5_1^k}) \leq 3$.

Similarly, suppose that we resolve any of the two precrossings so that petal 1 becomes the overstrand. 
Then, we can resolve the remaining two precrossings to allow petal 1 to be the overstrand and result in a trivial diagram.
On the other hand, we can resolve the precrossings to make petal 1 linked with petal 2 by ensuring that at least two of the adjacent precrossings are resolved with alternate resolutions, and obtain a nontrivial diagram. 
This proves that $kn(\overline{5_1^k}) > 2$ and  $tr(\overline{5_1^k})  > 2$, because resolving any two precrossings does not determine whether the resulting diagrams will be trivial or knotted.
Since $kn(\overline{5_1^k}) \leq 3$ and $kn(\overline{5_1^k}) > 2$, we have that $kn(\overline{5_1^k}) =3$.
Also, $tr(\overline{5_1^k})  > 2$ and the fact that the trivializing number of a diagram must always be even (by Lemma~\ref{tr1}) implies that $tr(\overline{5_1^k})  =4$. 
\end{ex}

\begin{ex}
Consider the projection $\overline{5_3^l}$ which is depicted in Figure~\ref{fig: 53l}, with labels 1 and 2 on its petals.

\begin{figure}[ht]
\[
\includegraphics[height=2.5cm]{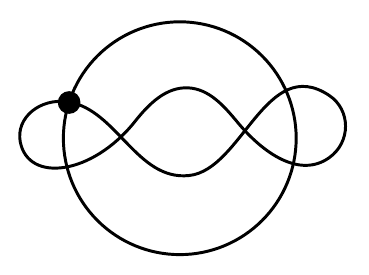}
\put(-53, 70){\fontsize{10}{11}2}
\put(-100, 35){\fontsize{10}{11}1}
\]
\caption{The Projection $\overline{5_3^l}$}
\label{fig: 53l}
\end{figure}

If we travel along petal 2 and resolve the two precrossings formed with petal 1 that lie on the right hand side of the diagram, so that the strand we travel along alternates between becoming an overstrand and an understrand, we obtain a nontrivial diagram because the petals are linked together.
Thus, $\overline{5_3^l}$ is nontrivial, $tr(\overline{5_3^l}) \neq 0$ and $kn(\overline{5_1^k}) \leq 2$. 
On the other hand, if we travel along petal 2 and resolve the precrossings so that petal 2 is always the overstrand, we obtain a trivial diagram. 
Therefore, we have that $tr(\overline{5_3^l}) \leq 2$. 

Since $tr(\overline{5_3^l}) \neq 0$, $tr(\overline{5_3^l}) \leq 2$, and the trivializing number must always be even, we must have that $tr(\overline{5_3^l}) =2$.
Additionally, since $kn(\overline{5_1^k}) \leq 2$ and the fact that $kn(P) \geq 2$ for any projection $P$ (by Lemma~\ref{kn1}), we know that $kn(\overline{5_1^k}) =2$. 
\end{ex} 
 
\subsection{Weighted Resolution Set}

Oyamaguchi proves that all of the prime 2-bouquets given in \cite{O} are distinct by calculating the Yamada polynomial~\cite{Y} of each of the 2-bouquets.
We would like to have a tool that can prove that all of the pseudodiagrams of prime 2-bouquets are distinct as well.
Such a tool that can distinguish between two objects is known as an \textbf{invariant}.
Before we can create an invariant to distinguish between pseudodiagrams of 2-bouquets, we need a few more definitions, which are inspired by techniques that Henrich et al. used in~\cite{He}. 

For the purposes of this section, we consider the equivalence classes of pseudodiagrams of 4-valent spatial graphs under the equivalence relation generated by the \textbf{extended pseudo-Reidemeister moves} depicted in Figure~\ref{fig:PReid}.
We refer to equivalence classes under this equivalence relation as \textbf{pseudographs}.
Since our research is focused on 2-bouquets, we refer to pseudographs that are 2-bouquets as \textbf{pseudo 2-bouquets}.

\begin{figure}[!ht]
\begin{center}
\includegraphics[height=4.5cm]{Reid.pdf}
\put(-242, 105){\fontsize{15}{11}RI:}
\put(-242, 60){\fontsize{15}{11}RII:}
\put(-242, 15){\fontsize{15}{11}RIII:}
\end{center}
\begin{center}
\includegraphics[height=1.5cm]{R4.pdf} 
\put(-250, 15){\fontsize{15}{11}RIV:}
\end{center}
\begin{center} 
\includegraphics[height=3.0cm]{R5.pdf} 
\put(-232, 60){\fontsize{15}{11}RV:}
\end{center}
\begin{center}
\includegraphics[height=4.5cm]{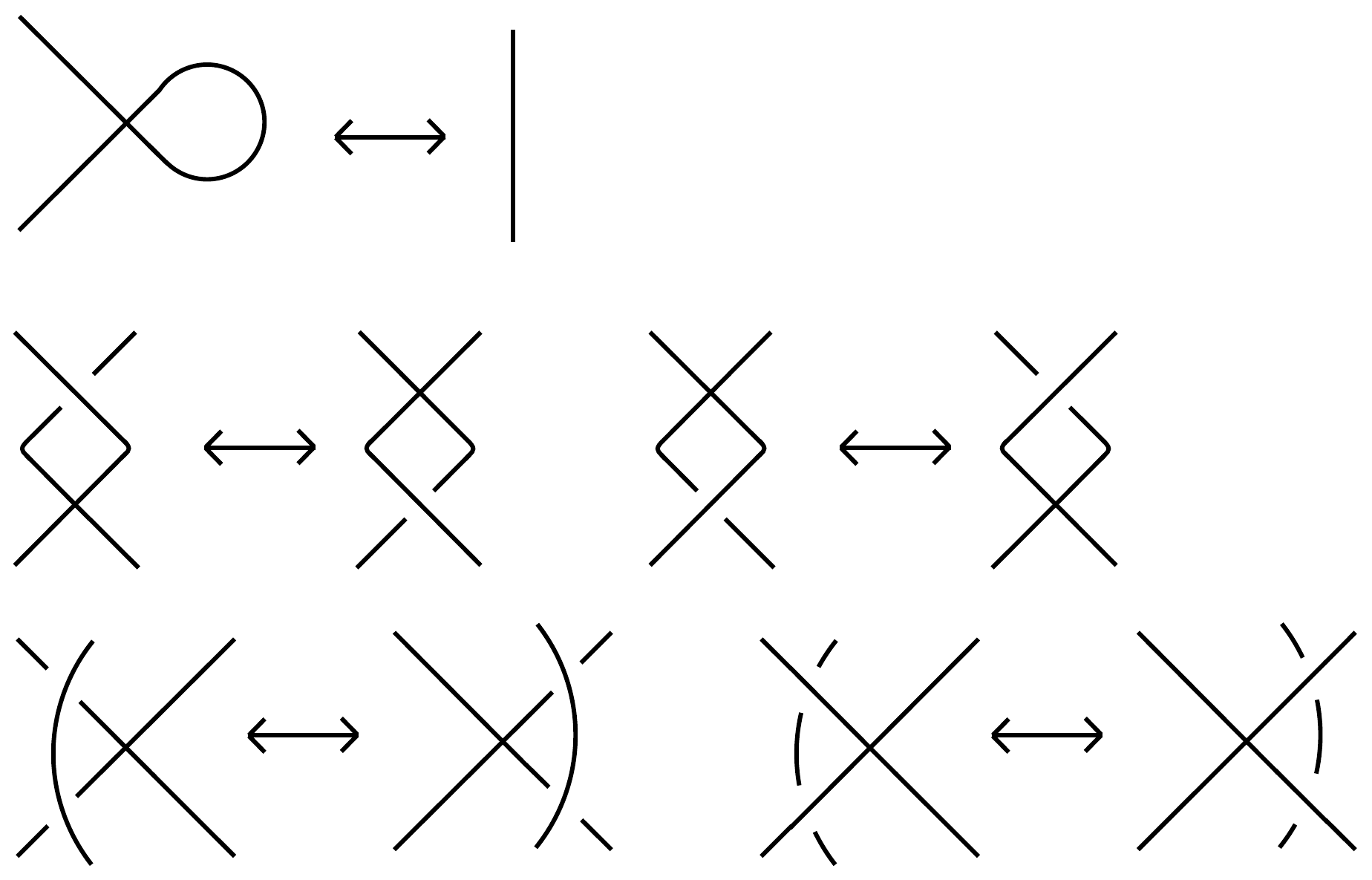}
\put(-240, 105){\fontsize{15}{11}PRI:}
\put(-240, 60){\fontsize{15}{11}PRII:}
\put(-240, 15){\fontsize{15}{11}PRIII:}
\caption{The Extended Pseudo-Reidemeister Moves}
\label{fig:PReid}
\end{center}
\end{figure}

The \textbf{weighted resolution set} of a pseudograph $\overline{G}$, denoted $C_W(\overline{G})$, is the set of ordered pairs $(G, m_G)$, where $G$ is a resolution of $\overline{G}$ and $m_G$ is the probability that $G$ is obtained from $\overline{G}$ by a random choice of crossing information for every precrossing, assuming that either resolution is equally likely. The multiset of all resolutions $G$ of $\overline{G}$ is considered up to the extended Reidemeister moves given in Figure~\ref{fig: Ext}.

\begin{rem}
Each resolution $G$ of $\overline{G}$ is a 4-valent spatial graph. 
\end{rem}

\begin{thm}
The weighted resolution set is an invariant of pseudographs.
\end{thm}

\begin{proof}
This proof is similar to that of Theorem 1 in~\cite{He}, but we provide it here to have a self contained paper. 

It suffices to show that the weighted resolution set of a pseudograph is unchanged by the extended pseudo-Reidemeister moves. 

First, all of the classical Reidemeister moves and the RIV and RV moves preserve the resolution multiset, since this multiset is considered up to the extended Reidemeister moves.

Next, consider the PRI move. 
Regardless of which resolution we use to resolve the precrossing, we can implement a Reidemeister I move and obtain a vertical strand.
Thus, the graph-type is preserved.
Every time we utilize the PRI move to add a precrossing, we increase the multiplicity of the diagram by two.
However, we have also increased the total number of diagrams that can be obtained by resolving all of the precrossings of the diagram by a factor of two. 
Therefore, the introduction of the precrossing in the PRI move does not affect the probabilities of the weighted resolution set. 

For the proof of the invariance of the weighted resolution set under the PRII and PRIII moves, we prove a single case for each of the moves, since the other cases are treated similarly.

Consider the PRII move. 
We obtain the diagrams shown in Figure~\ref{fig:PR2proof}, by resolving the precrossing in the two diagrams of the PRII move. 
In each case, there is a diagram that reduces to two vertical strands after a Reidemeister II move is applied, and a diagram that corresponds to two strands being linked in an identical way. 
We conclude that the weighted resolution set is unaffected by the PRII move. 

\begin{figure}[ht]
\begin{center}
\includegraphics[height=1.5cm]{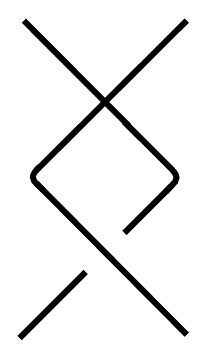}
\qquad
\qquad
\includegraphics[height=1.5cm]{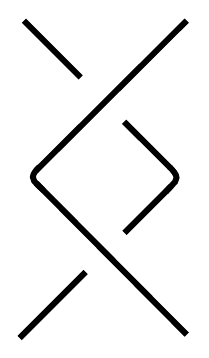}
\quad
\includegraphics[height=1.5cm]{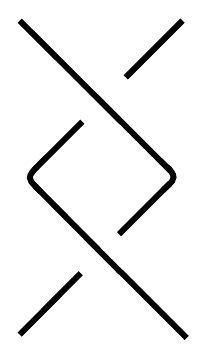}
\put(-100,15){\fontsize{50}{11}$\hookrightarrow$}
\end{center}
\begin{center}
\includegraphics[height=1.5cm]{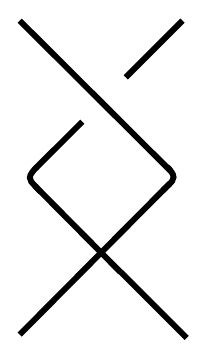}
\qquad
\qquad
\includegraphics[height=1.5cm]{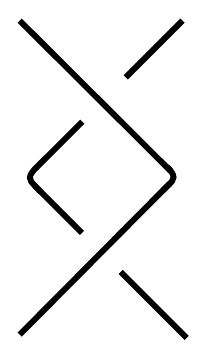}
\quad
\includegraphics[height=1.5cm]{PR2c.pdf}
\put(-100,15){\fontsize{50}{11}$\hookrightarrow$}
\caption{Proof of Invariance under PRII move}
\label{fig:PR2proof}
\end{center}
\end{figure}

Finally, consider the PRIII move.
After we resolve the precrossings in each of the diagrams involved in the PRIII move, we obtain the diagrams given in Figure~\ref{fig:PR3proof}.

\begin{figure}[ht]
\begin{center}
\includegraphics[height=1.5cm]{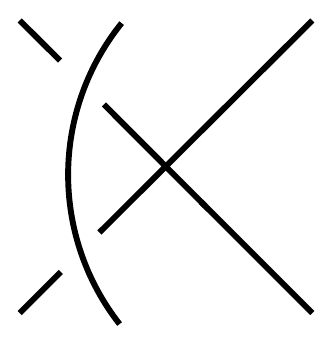}
\qquad
\qquad
\includegraphics[height=1.5cm]{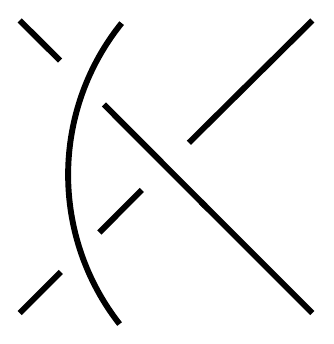}
\quad
\includegraphics[height=1.5cm]{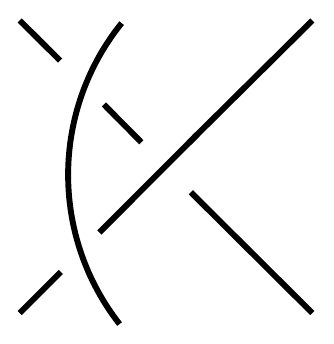}
\put(-135,15){\fontsize{50}{11}$\hookrightarrow$}
\end{center}
\begin{center}
\includegraphics[height=1.5cm]{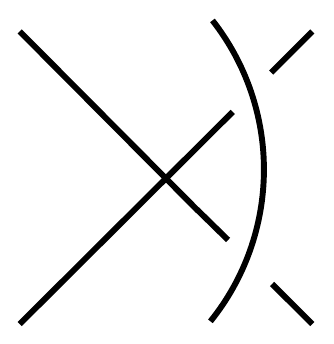}
\qquad
\qquad
\includegraphics[height=1.5cm]{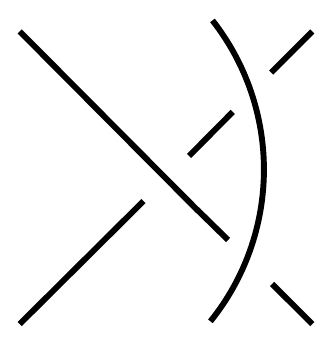}
\quad
\includegraphics[height=1.5cm]{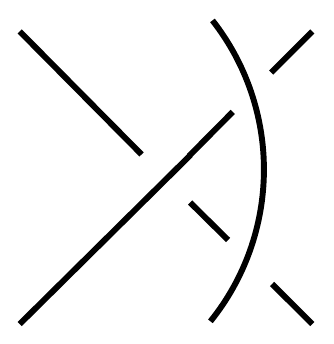}
\put(-135,15){\fontsize{50}{11}$\hookrightarrow$}
\caption{Proof of Invariance under PRIII move}
\label{fig:PR3proof}
\end{center}
\end{figure}

As seen in Figure~\ref{fig:PR3proof}, the resulting sets are the same, up to the Reidemeister III move.
Therefore, the weighted resolution set is unchanged under the application of a PRIII move. 
\end{proof}

\begin{cor}
The weighted resolution set is an invariant for pseudo 2-bouquets.
\end{cor}

\begin{ex}
In this example, we find the weighted resolution set of the pseudodiagram $Q$ in Figure~\ref{fig:41k}.
Note that this pseudodiagram is obtained by resolving one of the precrossings in the projection $\overline{4_1^k}$.

\begin{figure}[ht]
\[
\includegraphics[height=2.5cm]{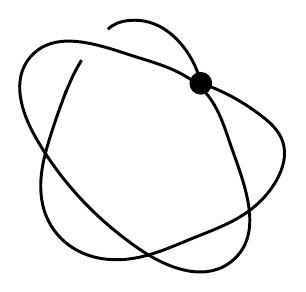}
\]
\caption{Pseudodiagram $Q$}
\label{fig:41k}
\end{figure}

Since there are three precrossings in Figure~\ref{fig:41k}, there are $2^3$ ways to resolve the precrossings in the diagram.
By considering all possible diagrams that result from resolving the precrossings, we obtain the following weighted resolution set,
\[
C_W(Q)=\left\{\left(0_1^k,\frac{3}{2^3}\right),\left(2_1^k, \frac{1}{2^3}\right), \left(2_1^{k*}, \frac{3}{2^3}\right),\left(4_1^{k*},\frac{1}{2^3} \right) \right\},
\]
where $G^*$ denotes the mirror image of $G$.
\end{ex}

With the invariant we have created, it is now easy to show that pseudodiagrams of prime 2-bouquets are distinct.
If we are given two pseudodiagrams and obtain different weighted resolution sets associated with these pseudodiagrams, we know that the pseudodiagrams are distinct. 
\bigskip

\noindent \textbf{Acknowledgments.}
The author would like to thank Dr. Carmen Caprau for her support and guidance through the writing and revision process of this paper. 
She is also grateful to Natsumi Oyamaguchi for providing and allowing her to use the diagrams of prime 2-bouquets shown in the Appendix. 

The author wishes to acknowledge the support from the Division of Graduate Studies at California State University, Fresno in the form of a Graduate Equity Fellowship. 


\clearpage

\section{APPENDIX}

The following is the list of prime 2-bouquets of type $K$ and $L$ up to six crossings as given in~\cite{O}. 
\begin{figure}[ht]
\[ \includegraphics[height=15.5cm]{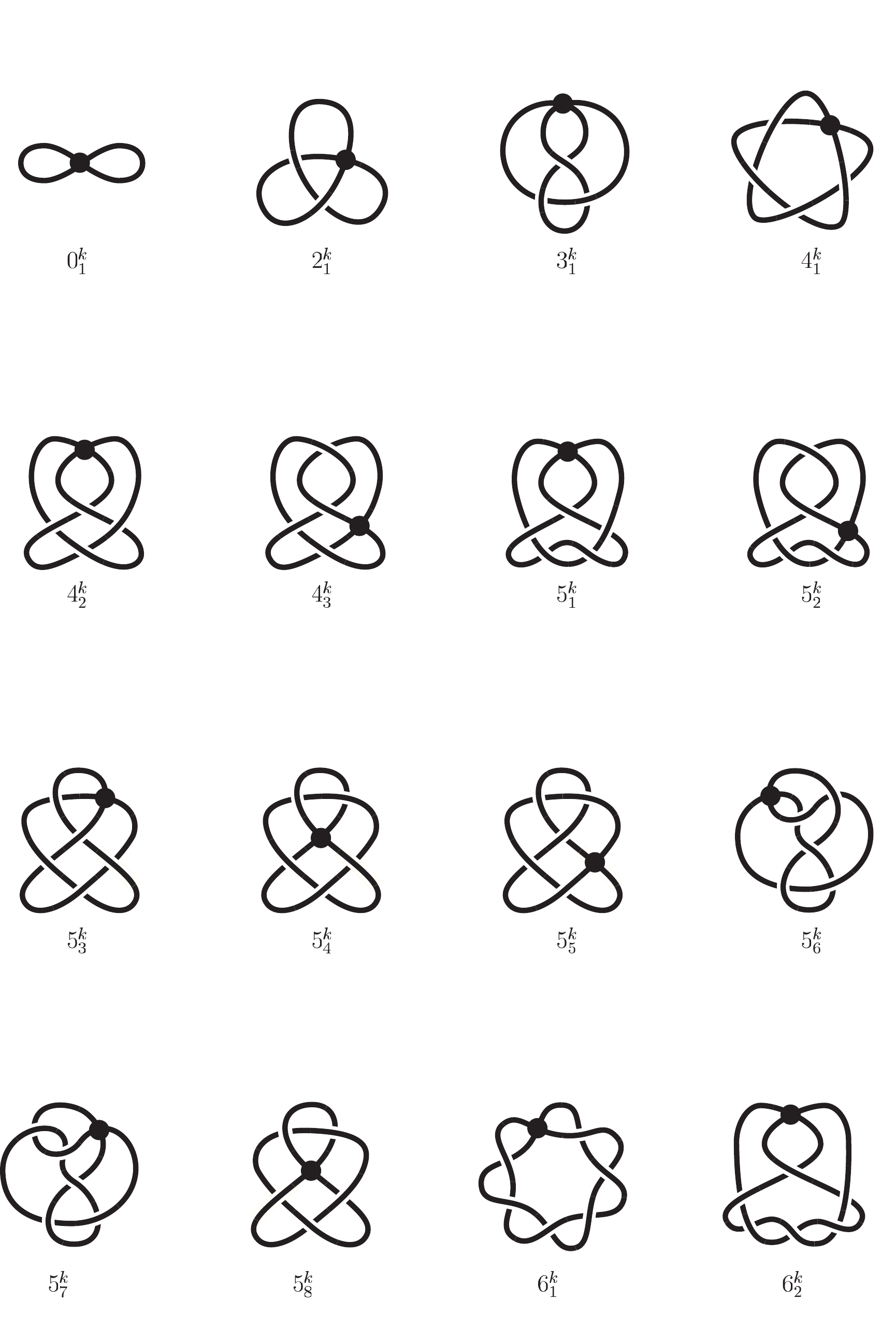} \]
\caption{Prime 2-Bouquets of Type $K$ up to Six Crossings}
\label{fig: N1}
\end{figure}

\begin{figure}[ht]
\[ \includegraphics[height=16.5cm]{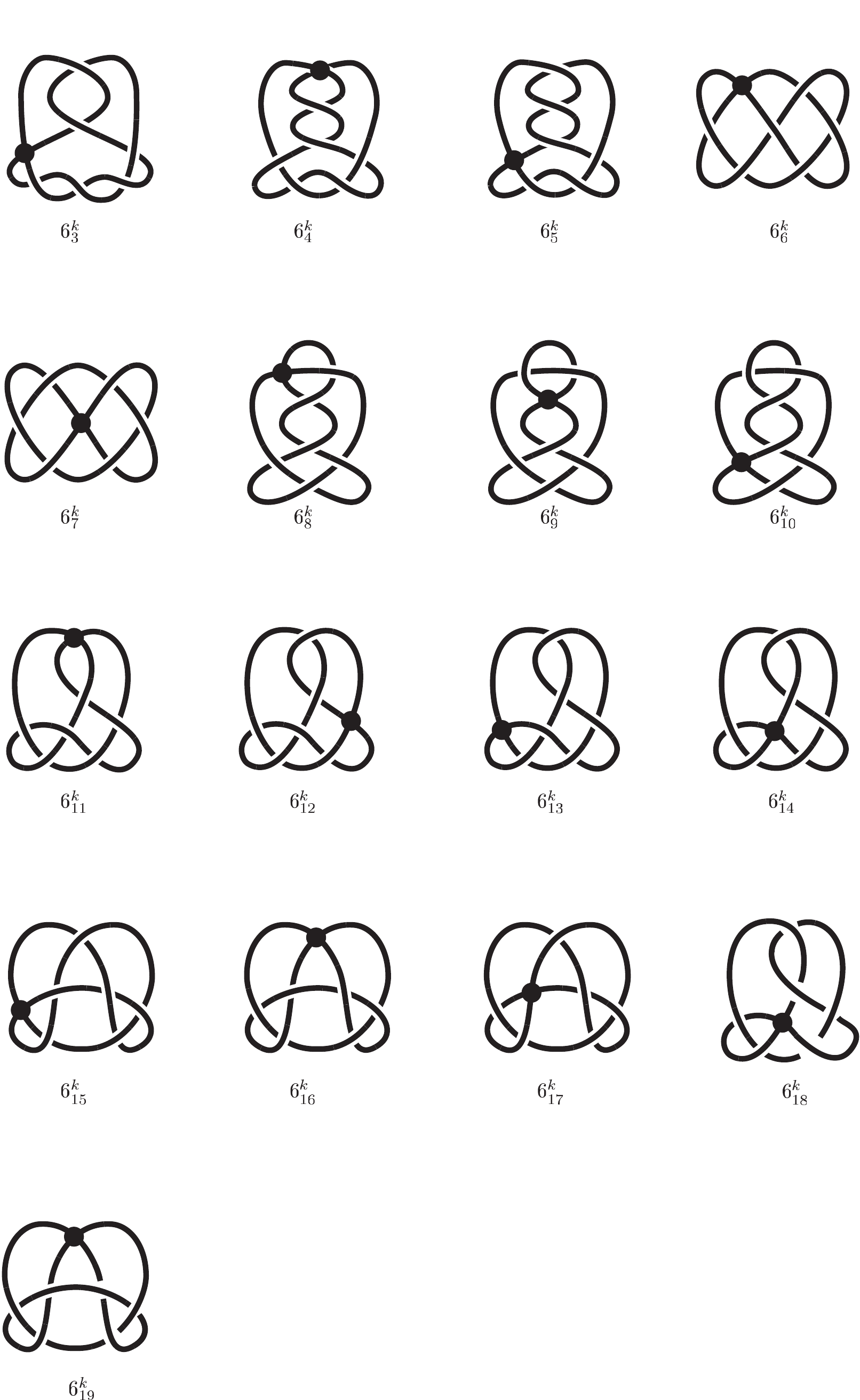} \]
\caption{Prime 2-Bouquets of Type $K$ up to Six Crossings (continued)}
\label{fig: N2}
\end{figure}

\begin{figure}[ht]
\[ \includegraphics[height=15.5cm]{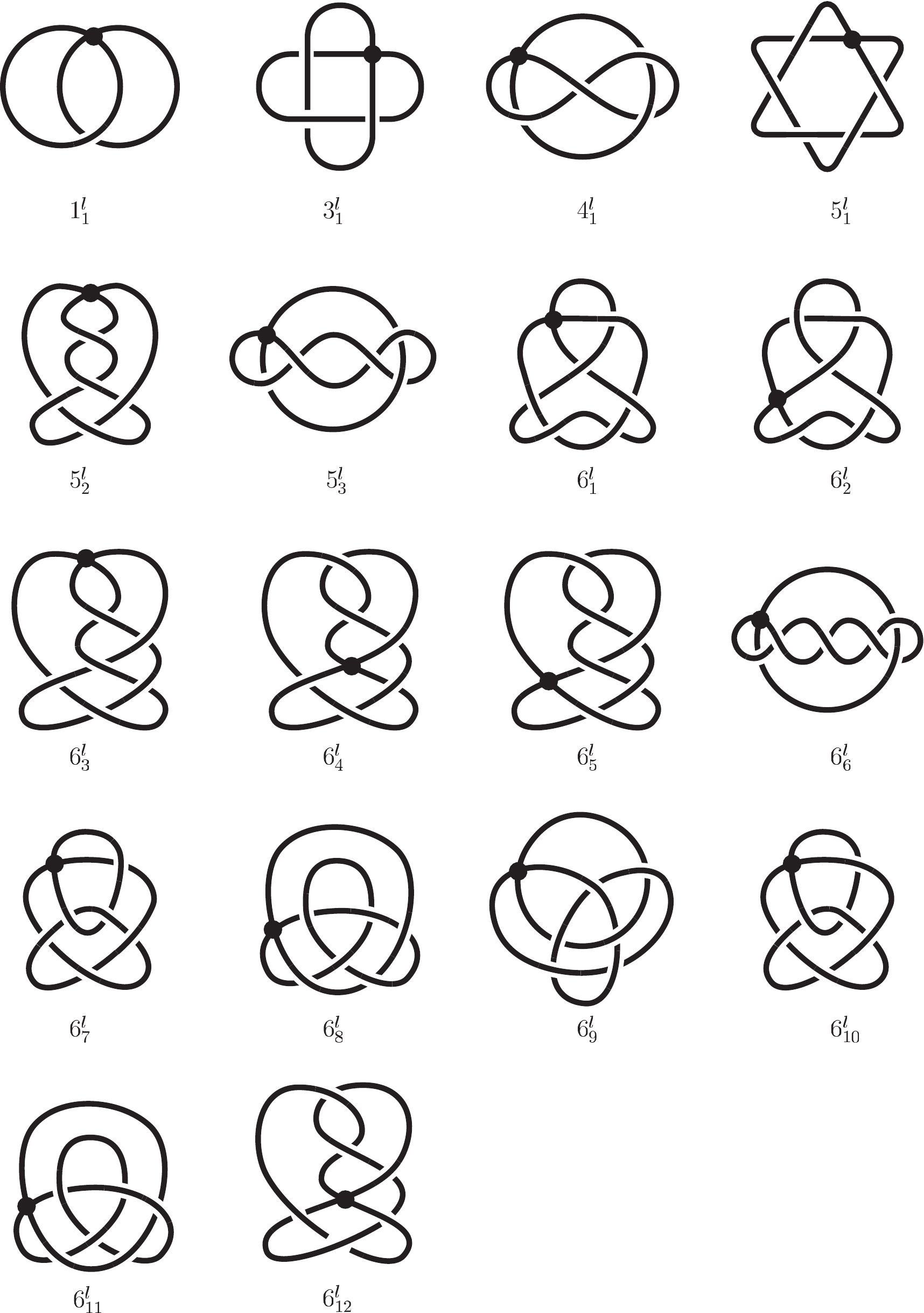} \]
\caption{Prime 2-Bouquets of Type $L$ up to Six Crossings}
\label{fig: N3}
\end{figure}

\end{document}